\documentclass[12pt]{article}
\usepackage{hyperref}
\hypersetup{
    pdftitle={3-degenerate induced subgraph of a planar graph},
    pdfauthor={Yangyan Gu, H. A. Kierstead, Sang-il Oum, Hao Qi, and Xuding Zhu}
}
\usepackage{float}
\usepackage{authblk}

\usepackage{graphicx}
\usepackage{subfigure}
\usepackage{tikz}
\usetikzlibrary{backgrounds,patterns,calc}
\usepackage[normalem]{ulem}
\usepackage[margin=3cm]{geometry}
\usepackage{enumitem}

\newlist{enum*}{enumerate*}{2}%
\setlist[enum*]{label={(\roman*)}}
\newlist{enum}{enumerate}{2}%
\setlist[enum]{label=(\alph*), nosep}
\setdescription{font=\it, leftmargin=0pt, itemsep=2pt, topsep=3pt,parsep=2pt, labelindent=2em, listparindent=2em}

\def\0{\emptyset}
\def\sub{\subseteq}

\newcommand\abs[1]{\lvert #1\rvert}

\usepackage{amsmath, amssymb, latexsym, amsthm, color}

\numberwithin{equation}{section}%

\newtheorem{guess}{Conjecture}[section]

\newtheorem{theorem}{Theorem}[section]

\newtheorem{lemma}[theorem]{Lemma}
\newtheorem{corollary}[theorem]{Corollary}

\newtheorem{observation}{Observation}[section]

\theoremstyle{definition}
\newtheorem{definition}{Definition}[section]

\newcommand{\dg}[2]{$(#1,\!#2)$-degenerate}
\newcommand{\ol}[1]{\overline{#1}}%
\newcommand{\inn}{\operatorname{int}}
\newcommand{\ext}{\operatorname{ext}}

\makeatletter
\providecommand*{\cupdot}{%
  \mathbin{%
    \mathpalette\@cupdot{}%
  }%
}
\newcommand*{\@cupdot}[2]{%
  \ooalign{%
    $\m@th#1\cup$\cr
    \sbox0{$#1\cup$}%
    \dimen@=\ht0 %
    \sbox0{$\m@th#1\cdot$}%
    \advance\dimen@ by -\ht0 %
    \dimen@=.5\dimen@
    \hidewidth\raise\dimen@\box0\hidewidth
  }%
}

\providecommand*{\bigcupdot}{%
  \mathop{%
    \vphantom{\bigcup}%
    \mathpalette\@bigcupdot{}%
  }%
}
\newcommand*{\@bigcupdot}[2]{%
  \ooalign{%
    $\m@th#1\bigcup$\cr
    \sbox0{$#1\bigcup$}%
    \dimen@=\ht0 %
    \advance\dimen@ by -\dp0 %
    \sbox0{\scalebox{2}{$\m@th#1\cdot$}}%
    \advance\dimen@ by -\ht0 %
    \dimen@=.5\dimen@
    \hidewidth\raise\dimen@\box0\hidewidth
  }%
}
\makeatother

\begin{document}

	\title{$3$-degenerate induced subgraph of a planar graph}
	
	\author[1]{Yangyan Gu}
	\author[2]{H. A. Kierstead}
	\author[3,4]{Sang-il Oum\thanks{Supported by the Institute for Basic Science (IBS-R029-C1)}}
	\author[5,6,1]{\\Hao Qi\thanks{Partially supported by research grants NSFC 11771403, NSFC 11871439 and MOST 107-2811-M-001-1534. }}
	\author[1]{Xuding~Zhu\thanks{Corresponding author. Partially supported by research grants NSFC 11971438,12026248, U20A2068.}}
	\affil[1]{Department of  Mathematics, Zhejiang Normal University, Jinhua,~China}
	\affil[2]{School of Mathematical and Statistical Sciences, Arizona~State~University, Tempe, AZ, USA}
	\affil[3]{Discrete Mathematics Group, Institute for Basic Science (IBS), Daejeon,~South Korea}
	\affil[4]{Department of Mathematical Sciences,  KAIST, Daejeon,~South Korea}
	\affil[5] {College of Mathematics and Physics, Wenzhou University, Wenzhou, China}
	\affil[6] {Institute of Mathematics, Academica Sinica, Taipei} 
	\affil[ ]{\small Email: \texttt{254945335@qq.com}, \texttt{kierstead@asu.edu},
	\texttt{sangil@ibs.re.kr},
	\texttt{qihao@wzu.edu.cn},~\texttt{xdzhu@zjnu.edu.cn}
	}

	\date{\today}
	
	\maketitle

	\begin{abstract}
		A graph $G$ is  $d$-degenerate  if every non-null subgraph of $G$ has a vertex of degree at most $d$.
		We prove that every $n$-vertex planar graph has a $3$-degenerate induced subgraph of order at least $3n/4$.
 
	\end{abstract}

\section{Introduction}	

	Graphs in this paper are simple, having no loops and no parallel edges. For a graph $G=(V,E)$, 
	the neighbourhood of $x\in V$ is 
	denoted by $N(x)=N_G(x)$, the degree of $x$ 
	is denoted by $d(x)=d_G(x)$, and the minimum degree of $G$ is denoted by $\delta(G)$. Let $\Pi=\Pi(G)$ be the set of total orderings of $V$. For $L\in \Pi$, we orient each edge  $vw\in E$ as $(v,w)$ if $w<_L v$ to form a directed graph $G_L$.  We denote  the \emph{out-neighbourhood}, also called the \emph{back-neighbourhood}, of $x$ by $N_G^L(x)$, the \emph{out-degree}, or \emph{back-degree}, of $x$  by $d_G^L(x)$. We write $\delta^+(G_L)$
	and $\Delta^+(G_L)$ to denote the minimum out-degree and the maximum out-degree, respectively, of $G_L$. We define $|G|:=|V|$, called 
	the \emph{order} of $G$, and $\|G\|:=|E|$.

	 An ordering $L\in\Pi(G)$     
	 is \emph{$d$-degenerate} if $\Delta^+(G_L)\le d$. A graph $G$  is $d$-degenerate if some $L\in\Pi(G)$ is $d$-degenerate. The  \emph{degeneracy} of $G$ is $\min_{L\in\Pi(G)}\Delta^+(G_L)$.   
	  It  is well known that  the degeneracy of 
	 $G$ is equal to $\max_{H\subseteq G}\delta(H)$.
 	
 	Alon, Kahn, and Seymour~\cite{aks} initiated the study of maximum $d$-degenerate
 	induced subgraphs in a general graph
 	and proposed the problem on planar graphs. 
	We study maximum  $d$-degenerate induced subgraphs of planar graphs. 
	For a non-negative integer $d$ and a graph $G$, let %
	\begin{align*}
	\alpha_d(G)&=\max\{|S|:S\subseteq V(G), ~ G[S] \text{ is $d$-degenerate}
	\}\text{ and}\\
 	\bar{\alpha}_d&= \inf \{\alpha_d(G)/\abs{V(G)}: G  \text{ is a non-null planar graph}
	\}.
	\end{align*}

        Let us review known bounds for $\bar\alpha_d$. Suppose that $G=(V,E)$ is a planar graph. For $d\ge 5$, trivially we have $\bar\alpha_d=1$ because planar graphs are $5$-degenerate.

        For $d=0$, a $0$-degenerate graph has no edges and therefore
	$\alpha_0(G)$ is the size of a maximum independent set of $G$. %
	By the Four Colour Theorem, $G$ has an  
	independent set $I$ with $|I|\ge\abs{V(G)}/4$. 
	Both $K_4$ and $C_8^2$ witness that $\bar\alpha_0\le 1/4$, so $\bar\alpha_0=1/4$.
         In 1968, Erd\H{o}s (see~\cite{Berge}) asked whether this bound  could be proved without the Four Colour Theorem. This question still remains open.
         In 1976, Albertson~\cite{Albertson1976} showed  that 
         $\bar\alpha_0 \ge 2/9$ independently of the Four Colour Theorem. This bound was improved to  $\bar\alpha_0 \ge 3/{13}$ independently of the Four Colour Theorem by Cranston and Rabern in 2016~\cite{Cranston}.

        For $d=1$, a $1$-degenerate graph is a
	forest. Since $K_4$ has no induced forest of
	order greater than $2$, we have $\bar\alpha_1 \le 1/2$.
	Albertson and Berman~\cite{AB} and Akiyama and Watanabe~\cite{AW1987} independently conjectured
	that $\bar\alpha_1 = 1/2$. In other words, every planar graph has an induced forest containing at least	half of its vertices.   This conjecture  	received much attention in the past 40 years; however, it remains largely open. 
	Borodin~\cite{borodin} proved that the vertex set of a planar graph
	can be partitioned
	into five classes such that the subgraph induced by the union of any two classes is   a
	forest. Taking the two largest classes yields an induced forest of
	order at least $2\abs{V(G)}/5 $. So $\bar\alpha_1 \ge 2/5$. This remains the best known lower bound on $\bar\alpha_1$. 
	On the other hand, the conjecture of Albertson and Berman, Akiyama and Watanabe was verified for some subfamilies of planar graphs. For example,    $C_3$-free, $C_5$-free,  or $C_6$-free planar graphs were shown in \cite{WL2002, FJMS2002} to be $3$-degenerate, and a greedy algorithm shows that the vertex set of a $3$-degenerate graph can be partitioned into two parts, each inducing a forest. 
	Hence $C_3$-free, $C_5$-free,  or $C_6$-free planar graphs satisfy the conjecture.
	Moreover, Raspaud and Wang~\cite{RW2008} showed that $C_4$-free planar graphs can be 
	partitioned into two induced forests,
	thus satisfying the conjecture.
	In fact, many of these graphs have larger induced forests. 
	Le~\cite{Le2018} showed that if a planar graph $G$ is $C_3$-free, then it has an induced forest with at least $5\abs{V(G)}/9$ vertices; 
	Kelly and Liu~\cite{KL2017} proved that if in addition $G$ is $C_4$-free, then 
	$G$ has an induced forest with at least $2\abs{V(G)}/3$ vertices.
		
        Now let us move on to the case that $d=2$. The octahedron  
has $6$ vertices and is $4$-regular, so  a   $2$-degenerate induced subgraph has at most $4$ vertices. Thus $\bar\alpha_2 \le 2/3$. We conjecture that equality holds. Currently, we only have a more or less trivial lower bound:  $\bar\alpha_2 \ge 1/2$, which follows from the fact that   $G$ is $5$-degenerate, and hence we can   
  greedily $2$-colour $G$ in an ordering that witnesses its degeneracy so that no vertex has three  out-neighbours of the same colour, i.e., each colour class induces a $2$-degenerate subgraph.    
  Dvo\v{r}\'{a}k and Kelly~\cite{DK2018} showed that 
  if a planar graph $G$ is $C_3$-free, 
  then it has a $2$-degenerate induced subgraph containing at least $4\abs{V(G)}/5$ vertices.

	For $d=4$, the icosahedron has $12$ vertices and is $5$-regular, so a  $4$-degenerate induced subgraph has at most $11$ vertices. Thus  $\bar\alpha_4 \le 11/12$. Again we conjecture that equality holds. The best known lower bound is $\bar\alpha_4 \ge 8/9$, which was obtained by 
	Luko\u{t}ka, Maz\'{a}k and Zhu~\cite{LMZ2014}.

	\medskip 

	In this paper, we study  $3$-degenerate induced subgraphs of planar graphs. Both the octahedron $C_6^2$ and
	the icosahedron witness that $\bar\alpha_3 \le 5/6$. 
    Here is our main theorem.   
	
	\begin{theorem}\label{main0}
		Every $n$-vertex planar graph has a
		 $3$-degenerate induced subgraph of order at least $3n/4$.
	\end{theorem}
	
	We conjecture that the upper bounds for $\bar\alpha_d$  mentioned above are tight.
	We remark that it is possible to obtain infinitely many $3$-connected tight examples for each $d$ by gluing together many copies of the tight example discussed above.
		
	\begin{guess}\label{tau}
		$ \bar\alpha_2=2/3, \bar\alpha_3=5/6$, and $\bar\alpha_4=11/12$.
	\end{guess}

The problem of colouring the vertices of a planar graph $G$
so that colour classes induce certain degenerate subgraphs has been studied in many papers. Borodin~\cite{borodin} proved that every planar graph $G$ is acyclically $5$-colourable, meaning that $V(G)$ can be coloured in $5$ colours so that a subgraph of $G$ induced by each colour class is $0$-degenerate and a subgraph of $G$ induced by the  union of any two colour classes  is $1$-degenerate. As a strengthening of this result, Borodin~\cite{borodin1976} conjectured that every planar graph has degenerate chromatic number at most $5$, which means that the vertices of any planar graph $G$ can be coloured in $5$ colours so that for each $i\in\{1,2,3,4\}$, a subgraph of $G$ induced by the union of any $i$ colour classes is $(i-1)$-degenerate.
This conjecture remains open, but it was proved in  \cite{KMSYZ2009}
that the list degenerate chromatic number of a graph is bounded by its $2$-colouring number, and it was proved in \cite{DKK2015} that the
$2$-colouring number of every planar graph is at most $8$. 
As consequences of the above conjecture, Borodin posed two other weaker conjectures:  (1) Every planar graph has a  vertex partition into two sets such that one induces a $2$-degenerate graph and the other induces a forest. (2) Every planar graph has a  vertex partition  into an independent set and a set inducing a $3$-degenerate graph.  
	Thomassen  confirmed these conjectures in \cite{thomassen1995} and \cite{tomassen2001}. 
	\medskip

This paper is organized as follows.
In Section~\ref{sec:notation} we will present our notation.
In Section~\ref{sec:main} we will  formulate a stronger theorem that allows us to apply induction.
This will involve identifying numerous obstructions to a more direct proof. 
In Section~\ref{sec:setup}, we will organize our proof by contradiction 
around the notion of an \emph{extreme counterexample}. 
In Sections~\ref{sec:separating}--\ref{sec:triangle}, 
we will develop properties of extreme counterexamples that eventually lead to a contradiction in Section~\ref{sec:proof}.

\section{Notation}\label{sec:notation}
For sets $X$ and $Y$, define %
  $Z=X\cupdot Y$ to mean $Z=X\cup Y$ and $X\cap Y=\0$.  
Let $G=(V,E)$ be a graph with $v,x,y\in V$ and $X, Y \subseteq V$. Then $\|v,X\|$ is the number of edges incident with $v$ and a vertex in $X$ and 
$
\|X,Y\|=\sum_{v\in X}\|v, Y\|
$.
When $X$ and $Y$ are disjoint, $\|X,Y\|$ is the number of edges $xy$ 
with $x\in X$ and  $y\in Y$. 
 In general, edges in $X\cap Y$ are counted twice by $\|X,Y\|$. Let $N(X)=\bigcup_{x\in X}N(x)-X$. 
 
 We write $H\sub G$ to indicate that $H$ is a subgraph of $G$. The subgraph of $G$ induced by a vertex set $A$ is denoted by $G[A]$. 
  The path $P$ with $V(P)=\{v_1,\ldots,v_n\}$ and $E(P)=\{v_1v_2,\ldots,v_{n-1}v_n\}$ is denoted by $v_1\cdots v_n$.  Similarly the cycle $C=P+v_nv_1$ is denoted by $v_1\cdots v_nv_1$. 

Now let $G$ be a simple connected plane graph. The boundary of the infinite face is denoted by $\mathbf{B}=\mathbf{B}(G)$ and $V(\mathbf{B}(G))$ is denoted by $B=B(G)$.  Then $\mathbf B$ is a subgraph of the outerplanar graph $G[B]$. For a %
cycle  $C$ in $G$, let $\inn_G[C]$ denote the subgraph of $G$ obtained by removing all exterior vertices and edges 
and let $\ext_G[C]$ be the subgraph of $G$ obtained by removing all interior vertices and edges. Usually the graph $G$ is clear from the text,  and we write $\inn[C]$ and $\ext[C]$ for $\inn_G[C]$ and $\ext_G[C]$. 
Let $\inn (C) = \inn [C] -V(C)$ and $\ext (C) = \ext[C]-V(C)$. Let 
 $N^{\circ}(x)=N(x)-B$ and $N^{\circ}(X)=N(X)-B$.

For $L\in \Pi$, the \emph{up-set} of $x$ in $L$ is defined as $U_L(x)=\{y\in V:y>_Lx\}  $ and the \emph{down-set} of $x$ in $L$ is defined as $D_L(x)=\{y\in V:y<_Lx\}  $.
Note that for each $L\in \Pi$, $y<_L x$ means that $y\le_L x$ and $y\neq x$.
For two sets $X$ and $Y$, we say $X\le_L Y$ if $x\le_L y$ for all $x\in X$, $y\in Y$. 

\section{Main result}\label{sec:main}
 
In this section we phrase a stronger, more technical version of Theorem~\ref{main0} that is more amenable  to induction. This  is roughly analogous to the proof of the $5$-Choosability Theorem by Thomassen~\cite{tomassen1994}. 

If $G=G_1\cup G_2$ and $G_1\cap G_2=G[A]$ for a set $A$ of vertices, then we would like to join two $3$-degenerate subgraphs obtained from $G_1$ and $G_2$ by induction to form a $3$-degenerate subgraph of~$G$. The problem is that vertices from~$A$ may have neighbours in both subgraphs.
Dealing with this motivates the following definitions.

\medskip

Let $A\subseteq V(G)$.
A subgraph $H$ of $G$ is \emph{\dg{k}{A}}
if there exists an ordering $L\in \Pi(G)$ such that
$A\le_L V-A$ and $d_H^L(v)\le k$ for every vertex $v\in V(H)-A$.
Equivalently, every subgraph $H'$ of $H$ with $V(H')- A\neq \emptyset$
has a vertex $v \in V(H')-A$ such that $d_{H'}(v)\le k$.
A subset $Y$ of $V$ is \emph{$A$-good} if $G[Y]$    is  \dg{3}{A}. 
We say a subgraph $H$ is $A$-good if $V(H)$ is $A$-good.
Thus if $A =\emptyset$ then $G$ is $A$-good if and only if $G$ is $3$-degenerate. 
Let
\[f(G; A) = \max\{\abs{Y}: ~ Y\subseteq V(G) \text{ is $A$-good} \}.\]
Since $\emptyset$ is $A$-good, $f(G;A)$ is well defined.

For an induced subgraph $H$ of $G$ and a set $Y$ of vertices of $H$, we say $Y$ is \emph{collectable} in $H$ if the vertices of $Y$ can be ordered as $y_1,y_2,\ldots,y_k$ such that for each $i\in \{1,2,\ldots,k\}$, 
either $y_i \notin A$ and $d_{H-\{y_1,y_2,\ldots,y_{i-1}\}}(y_i)\le 3$
or $V(H)-\{y_1,y_2,\ldots,y_{i-1}\} \subseteq A$.

In order to  build an $A$-good subset, we typically apply a sequence of operations of deleting and collecting.
\emph{Deleting $X \subset V$}  means   replacing $G$ with  $G-X$. 
An ordering witnessing that $Y$ is collectable is called a \emph{collection}  order. 
For disjoint subsets $V_1, \dots, V_s$ of $V$, if $V_{i}$ is collectable in $G-\bigcup_{j=1}^{i-1} V_j$ for each $i=1,2,\ldots,s$, then 
\emph{collecting} $V_1,\dots, V_s$ means first putting $V_1$ at the end of $L$ in a collection order for $V_1$, then putting $V_2$ at the end of $L-V_1$ in a collection order for $V_2$ in $G-V_1$, etc.  Note that if $Y$ is a collectable set in $G$ and $V-Y$ is $A$-good, then $V$ is $A$-good.

\begin{definition}\label{admissible}
	A path $v_1v_2\ldots v_\ell$ of a plane graph $G$ is \emph{admissible}
	if $\ell>0$ and it is a path in $\mathbf B(G)$
	such that for each $1<i<\ell$, $G-v_i$ has no path from $v_{i-1}$ to $v_{i+1}$.
\end{definition}
 
A path of length $0$ has only $1$ vertex in its vertex set.

\begin{definition}\label{usable}
	A set $A$ of vertices of a plane graph $G$ is \emph{usable} in $G$ if
	for each component $G'$ of $G$, $A\cap V(G')$ is the empty set or
	the vertex set of an admissible path of $G'$.
\end{definition}

\begin{lemma}\label{Latmost2}
Let $G$ be a plane graph and let $A$ be a usable set in $G$. Then for each vertex $v$ of $G$, $|N_G(v) \cap A| \le 2$. 
\end{lemma}
\begin{proof}
    This is clear from the definition of an admissible path. 
\end{proof}

\begin{observation}\label{ob0}	
	If $G$ is outerplanar and $A$ is a usable set in $G$, then $G$ is \dg{2}{A}.
\end{observation}
 
Observation \ref{ob0} motivates the expectation that plane graphs with large boundaries have large $3$-degenerate induced subgraphs.
Roughly, we intend to prove that $f(G;A)\le 3\abs{V(G)}/4+\abs{B}/4$. This formulation provides a potential function for measuring progress as we collect and delete vertices. 
For example, deleting a boundary vertex with at least four interior neighbours 
provides a smaller graph whose potential is at least as large. 
Some of the bonus $\abs{B}/4$ is needed for dealing with chords. 
But this does not quite work; $C_6^2$ is a counterexample, and there are 
infinitely many more. The rest of this section is devoted to formulating a more refined potential function.

A set $Z$ of vertices is said to be \emph{exposed} if $Z\sub B$.  
We say that a vertex $z$ is \emph{exposed} if $\{z\}$ is exposed.
We say that deleting $Y$ and collecting $X$ \emph{exposes} $Z$ if $Z\sub B(G-Y-X)-B$.   

\begin{definition}
	\label{def-special}
	Let 
	$\mathcal{Q}=\{Q_{1},Q_{2},Q_{2}^{+},Q_{3}, Q_{4},Q_{4}^{+},Q_{4}^{++}\}$ be the set of plane graphs shown in Figure~\ref{octahedron}. 
	For a plane graph $G$, a cycle  $C$ of $G$ is \emph{special} if $G_{C}:=\inn_{G}[C]$ is isomorphic to a plane graph in $\mathcal Q$, where $C$ corresponds to the boundary. 
	In this case, $G_{C}$ is also \emph{special}. 
\end{definition}	
	For a special cycle $C$ of a plane graph $G$, we define 
	\begin{align*}
		T_{C}&:=\inn_{G}(C), \text{ which is isomorphic to $K_3$,}\\
		X_{C}&:=\{v\in V(C):  \text{there is a facial cycle $D$   
		such that}\\
		&\qquad\qquad\qquad\qquad \text{$v\in V(C)\cap V(D)$ and $\abs{V(T_C)\cap V(D)}=2$}\},\\
		V_{C}&:=V(G_{C}),~ Y_C:=X_C \cup V(T_C), \text{ and }\ol Y_{C}:=V_{C}-Y_{C}=V(C)-X_C.
	\end{align*}
	Then $V(C)=X_{C}\cup\ol Y_{C}$.

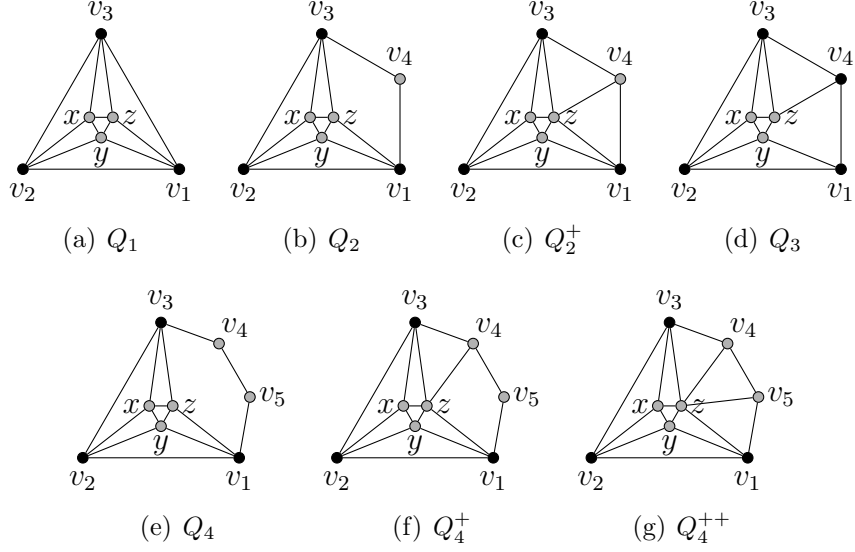
\begin{figure}
	\centering
	\tikzstyle{xc}=[circle,draw,fill=black,inner sep=0pt,minimum width=4pt]
	\tikzstyle{every node}=[circle,draw,fill=black!30,inner sep=0pt,minimum width=4pt]
	\subfigure[{$Q_1$}]{
	   \label{q1} %
	   \begin{tikzpicture}[scale=1.2]
	   \node [xc,label=$v_3$] at (90:1) (x) {};
	   \node [xc,label=below:$v_2$] at (210:1)(y) {};
	   \node [xc,label=below:$v_1$] at (330:1)(z) {};
	   \node  [label=left:$x$] at (150:.15) (v0) {};
	   \node  [label=below:$y$] at (150+120:.15) (v1) {};
	   \node  [label=right:$z$] at (150+240:.15) (v2) {};
	   \draw (v1)--(v2)--(v0)--(v1);
	   \draw (v0)--(y)--(v1)--(z)--(v2)--(x)--(v0);
	   \draw (x)--(y)--(z)--(x);
	   \end{tikzpicture}
   }
   \subfigure[$Q_2$]{
	   \label{q2} %
	   \begin{tikzpicture}[scale=1.2]
	   \tikzstyle{s}=[rectangle,draw,fill=black,inner sep=0pt,minimum width=4pt,minimum height=4pt]
	   \node [xc,label=$v_3$] at (90:1) (x) {};
	   \node [xc,label=below:$v_2$] at (210:1)(y) {};
	   \node [xc,label=below:$v_1$] at (330:1)(z) {};
	   \node [label=$v_4$] at (30:1) (u){};
	   \node  [label=left:$x$] at (150:.15) (v0) {};
	   \node  [label=below:$y$] at (150+120:.15) (v1) {};
	   \node  [label=right:$z$] at (150+240:.15) (v2) {};				\draw (v1)--(v2)--(v0)--(v1);
	   \draw (v0)--(y)--(v1)--(z)--(v2)--(x)--(v0);
	   \draw (x)--(y)--(z);
	   \draw (x)--(u)--(z);
	   \end{tikzpicture}
   }
   \subfigure[$Q_2^+$]{
	   \label{q2+} %
	   \begin{tikzpicture}[scale=1.2]
	   \tikzstyle{s}=[rectangle,draw,fill=black,inner sep=0pt,minimum width=4pt,minimum height=4pt]
	   \node [xc,label=$v_3$] at (90:1) (x) {};
	   \node [xc,label=below:$v_2$] at (210:1)(y) {};
	   \node [xc,label=below:$v_1$] at (330:1)(z) {};
	   \node [label=$v_4$] at (30:1) (u){};
	   \node  [label=left:$x$] at (150:.15) (v0) {};
	   \node  [label=below:$y$] at (150+120:.15) (v1) {};
	   \node  [label=right:$z$] at (150+240:.15) (v2) {};				\draw (v1)--(v2)--(v0)--(v1);
	   \draw (v0)--(y)--(v1)--(z)--(v2)--(x)--(v0);
	   \draw (x)--(y)--(z);
	   \draw (x)--(u)--(z);
	   \draw (u)--(v2);
	   \end{tikzpicture}
   }
   \subfigure[$Q_3$]{
	   \label{q3} %
	   \begin{tikzpicture}[scale=1.2]
	   \tikzstyle{s}=[rectangle,draw,fill=black,inner sep=0pt,minimum width=4pt,minimum height=4pt]
	   \node [xc,label=$v_3$] at (90:1) (x) {};
	   \node [xc,label=below:$v_2$] at (210:1)(y) {};
	   \node [xc,label=below:$v_1$] at (330:1)(z) {};
	   \node [xc,label=$v_4$] at (30:1) (u){};
	   \node  [label=left:$x$] at (150:.15) (v0) {};
	   \node  [label=below:$y$] at (150+120:.15) (v1) {};
	   \node  [label=right:$z$] at (150+240:.15) (v2) {};
	   \draw (v1)--(v2)--(v0)--(v1);
	   \draw (v0)--(y)--(v1)--(z);
	   \draw (v2)--(x)--(v0);
	   \draw (x)--(y)--(z);
	   \draw (x)--(u)--(z);
	   \draw (v2)--(u);
	   \end{tikzpicture}
   } \\
	   \subfigure[$Q_4$]{
		   \label{q4} %
		   \begin{tikzpicture}[scale=1.2]
		   \node [xc,label=$v_3$] at (90:1) (x) {};
		   \node [xc,label=below:$v_2$] at (210:1)(y) {};
		   \node [xc,label=below:$v_1$] at (330:1)(z) {};
		   \node [label=right:$v_5$] at (10:1)(v) {};
		   \node [label=above right:$v_4$] at (50:1) (u){};
		   \node  [label=left:$x$] at (150:.15) (v0) {};
		   \node  [label=below:$y$] at (150+120:.15) (v1) {};
		   \node  [label=right:$z$] at (150+240:.15) (v2) {};
		   \draw (v1)--(v2)--(v0)--(v1);
		   \draw (v0)--(y)--(v1)--(z)--(v2)--(x)--(v0);
		   \draw (x)--(y)--(z);
		   \draw (x)--(u)--(v)--(z);
		   \end{tikzpicture}
	   }
	   \subfigure[$Q_4^+$]{
		   \label{q4+} %
		   \begin{tikzpicture}[scale=1.2]
		   \node [xc,label=$v_3$] at (90:1) (x) {};
		   \node [xc,label=below:$v_2$] at (210:1)(y) {};
		   \node [xc,label=below:$v_1$] at (330:1)(z) {};
		   \node [label=right:$v_5$] at (10:1)(v) {};
		   \node [label=above right:$v_4$] at (50:1) (u){};
		   \node  [label=left:$x$] at (150:.15) (v0) {};
		   \node  [label=below:$y$] at (150+120:.15) (v1) {};
		   \node  [label=right:$z$] at (150+240:.15) (v2) {};
		   \draw (v1)--(v2)--(v0)--(v1);
		   \draw (v0)--(y)--(v1)--(z)--(v2)--(x)--(v0);
		   \draw (x)--(y)--(z);
		   \draw (x)--(u)--(v)--(z);
		   \draw (u)--(v2);
		   \end{tikzpicture}
	   }
	   \subfigure[$Q_4^{++}$ ]{
		   \label{q4++} %
		   \begin{tikzpicture}[scale=1.2]
		   \node [xc,label=$v_3$] at (90:1) (x) {};
		   \node [xc,label=below:$v_2$] at (210:1)(y) {};
		   \node [xc,label=below:$v_1$] at (330:1)(z) {};
		   \node [label=right:$v_5$] at (10:1)(v) {};
		   \node [label=above right:$v_4$] at (50:1) (u){};
		   \node  [label=left:$x$] at (150:.15) (v0) {};
		   \node  [label=below:$y$] at (150+120:.15) (v1) {};
		   \node  [label=right:$z$] at (150+240:.15) (v2) {};
		   \draw (v1)--(v2)--(v0)--(v1);
		   \draw (v0)--(y)--(v1)--(z)--(v2)--(x)--(v0);
		   \draw (x)--(y)--(z);
		   \draw (x)--(u)--(v)--(z);
		   \draw (u)--(v2)--(v);
		   \end{tikzpicture}
	   }
   \caption{Plane graphs in $\mathcal Q$ defining special subgraphs $G_C$ where $C$ corresponds to the boundary cycle.
   Solid black vertices denote vertices in $X_C$.}
   \label{octahedron} %
\end{figure}

\begin{observation}
	\label{obs1}
	Let $A$ be a usable set in a plane graph $G$.
	Let $C=v_{1}\dots v_{k}v_{1}$ be a special cycle of $G$.
	If $G_C$ is (not only isomorphic but also equal to a plane graph) in  $\mathcal Q$, then the following hold.
	\begin{enum}
		\item $T_{C}=xyzx$ with $N_G(x)=\{y,z,v_{2},v_{3}\}$ and $N_{G}(y)=\{x,z,v_{1},v_{2}\}$.
		\item $X_{C}=\{v_1,v_2,v_3\}$ if $G_C\neq Q_3$ and $X_C=\{v_1,v_2,v_3,v_4\}$ if $G_C=Q_3$.
	    \item Deleting any vertex in $X_C \cap B$ exposes two vertices of $T_C$.
		\item\label{obs1collect} For each vertex $v\in X_C$, $V(T_{C})$ is collectable in $G-v$, 
		except that if $G_{C}=Q_{4}^{++}$ and $v=v_{2}$ then only $\{x,y\}$ is collectable in $G-v$.
		\item\label{obs1uniq} If $\ol Y_{C}\ne\emptyset$ then 
		there is a facial cycle $C^*$ containing $\ol Y_C\cup \{v\}$ for some $v\in V(T_C)$. Moreover, $v=z$ is unique, and 
		if $|\ol Y_C|=2$, then $C^{*}$ is unique. 
	\item $T_C$ has at least two vertices $v$ such that $d_G(v)=4$.
	\end{enum}  
\end{observation}  

Note that vertices on $C$ may have neighbours in $\ext(C)$ or maybe contained in $A$. Thus we may not be able to collect vertices of $C$.

A special cycle $C$ is called \emph{exposed} if $X_C \subseteq   B(G)$.
A \emph{special cycle packing} of $G$ is a set of 
exposed special cycles $\{C_1, \ldots, C_m\}$ such that $Y_{C_i}\cap Y_{C_j}=\emptyset$ for all $i\neq j$. 
Let 
$\tau(G)$ be the maximum cardinality of a special cycle packing
and 
\[\partial(G) =  \frac{3}{4}|V(G)|+\frac{1}{4}(|B| -  \tau(G)).\]
We say that a special cycle packing of $G$ is \emph{optimal} if its cardinality is equal to $\tau(G)$.
 
\begin{theorem}\label{main}
    For all  plane graphs $G$ and usable sets $A\subseteq B(G)$,
    \begin{equation}\label{mb}
    f(G;A)  \ge \partial(G).
    \end{equation}
\end{theorem}
 
 Clearly $|B| -\tau(G) \ge 2$ for any plane graph $G$ with at least $2$ vertices. 
 This is trivial if $\tau(G)=0$. If $\tau(G)=k$, then each of the $k$ exposed cycles in the maximum cardinality special cycle packing of $G$ has at least $3$ vertices in $B$ and therefore $|B|-\tau(G)\ge 2k\ge 2$.
 The following consequence of Theorem~\ref{main} is the main result of this paper.
 
 \begin{corollary}
 	Every  $n$-vertex planar graph $G$ (with $n \ge 2$) has an induced $3$-degenerate subgraph $H$ with $|V(H)| \ge (3n+2)/4$. 
 \end{corollary}
 
 \section{Setup of the proof}\label{sec:setup}

 Suppose Theorem~\ref{main} is not true.
 Among all counterexamples, choose $(G;A)$ so that
\begin{enum}
\item [(i)]  $|V(G)|$ is minimum,
\item [(ii)] subject to (i), $|A|$ is maximum, and
\item [(iii)] subject to (i) and (ii), $\abs{E(G)}$ is maximum.
\end{enum}
We say that such a counterexample is \emph{extreme}.

If $A'\nsubseteq V(G')$, then
we may abbreviate $(G';A'\cap V(G'))$ by $(G';A')$, but still (ii) refers to $|A'\cap V(G')|$.
We shall derive a sequence of properties of $(G;A)$
that leads to a contradiction.
Trivially $\abs{V(G)}>2$,  $G$ is connected (if $G$ is the disjoint union of $G_1$ and $G_2$, then 
$f(G;A) = f(G_1;A)+f(G_2;A)$ and $\partial(G)=\partial(G_1)+\partial(G_2)$).

  \begin{lemma}\label{lem:reduction}
  	Let $G$ be a plane graph and $X$ be a subset of $V(G)$.
  	If $A$ is usable in $G$, then
  	$A-X$ is usable in $G-X$.
  \end{lemma}
  \begin{proof}
  	We may assume that $G$ is connected and $X=\{v\}$. If $v\notin A$, then it is trivial. 
  	Let $P=v_0v_1\cdots v_k$ be the admissible path in $G$ such that $A=V(P)$.
  	If $v=v_0$ or $v=v_k$, then again it is trivial. 
  	If $v=v_i$ for some $0<i<k$, then by the definition of admissible paths, 
  	$G-v_i$ is disconnected, and $v_{i-1}$ and $v_{i+1}$ are in distinct components. 
  	Thus again $A-\{v\}$ is usable in $G-v$.
  \end{proof}
  
  Suppose $Y$ is a nonempty subset of $V(G)$ and $G[Y]$ is connected. Let   $C$ be an exposed special cycle of $G'=G-Y$. Then  $C$ satisfies one of the following conditions. 
  \begin{enumerate}
	\item $C$ is an exposed special cycle of $G$. 
	\item $C$ is a non-exposed special cycle of $G$; in this case
	  $X_{C} \cap (B(G')-B) \ne \emptyset$.
	\item $C$ is not a special cycle of $G$; in this case 
	$Y\subseteq\inn_{G}(C)$, and so $Y\cap B=\emptyset$.
\end{enumerate}
A cycle $C$ is \emph{type-a}, \emph{-b}, \emph{-c}, respectively, if it satisfies condition (a), (b), (c), respectively.
Let 
\[\delta(Y) = \begin{cases} 1, & \text{ if $G'$ has a  type-c exposed special cycle},\\
0, & \text{ otherwise}.
\end{cases}
\]

\begin{lemma}
\label{lem-delta}
Let $Y$ be a nonempty subset of $V(G)$ such that $G[Y]$ is connected. Let $G'=G-Y$. If $C, C'$ are distinct exposed type-c special cycles of $G'$, then $Y_C \cap Y_{C'} \ne \emptyset$.
\end{lemma}
\begin{proof}
	Let $C$, $C'$ be distinct exposed type-c special cycles of $G'$. 
	Since $B\cap Y=\emptyset$ and $G[Y]$ is connected,
	there exists a facial cycle $D$ of $G'$ such that $\inn_G(D) = G[Y]$. 
	Then  $D$  is a facial cycle of both $G'_{C}$ and $G'_{C'}$.
	Arguing by contradiction, suppose $Y_C \cap Y_{C'} = \emptyset$. 
	Since $V(D)\subseteq V(G'_C)=Y_C\cup \ol Y_C$ and 
	$V(D)\subseteq V(G'_{C'})=Y_{C'}\cup \ol Y_{C'}$, 
	we have \[V(D)\subseteq V(G'_C)\cap V(G'_{C'})\subseteq \ol Y_C \cup \ol Y_{C'}.\]
	By symmetry we may assume that $|\ol Y_{C}\cap V(D)|\ge |\ol Y_{C'}\cap V(D)|$.	
	Using $|\ol Y_{C}|, |\ol Y_{C'}|\le 2$,  
	we deduce that 
	\[3\le |V(D)|\le  |V(G'_C)\cap V(G'_{C'})|\le 4, ~
	\ol Y_C\subseteq V(D),~	|\ol Y_{C}|=2,
	\text{ and } \ol Y_{C'}\cap V(D)\neq\emptyset.\] 

	We will show that $H$ is isomorphic to $H_{1}$ or $H_{2}$ in 
	Figure~\ref{figA}. Since $|\ol Y_{C}|=2$, 
	by Observation~\ref{obs1}\ref{obs1uniq},
	$D$ is the unique facial cycle
	in $G'_{C}$ such that 
	there is a vertex  $\dot z\in V(T_{C})$ with $\ol Y_{C}\cup \{\dot z\}\sub V(D)$. 
	As $\dot z\in V(D)$ and $\dot z\in Y_{C}$, we have $\dot z\in \ol Y_{C'}$. 
	Since $\ol Y_{C'}\neq\emptyset$, 
	again by Observation~\ref{obs1}\ref{obs1uniq}, 
	there is a unique vertex $\ddot z\in V(T_{C'})$ such that $\ol Y_{C'}\cup \{\ddot z\}$ is
	contained in a facial cycle of $G'_{C'}$. Then $\ddot z\in Y_{C'}\cap \ol Y_{C}\sub V(D)$. 

	First we show that $|V(G'_C)\cap V(G'_{C'})|=4$.  Assume to the contrary that  
	$|V(G'_C)\cap V(G'_{C'})|=3$. Since 
	$V(D) \subseteq V(G'_C)\cap V(G'_{C'})$, we conclude that  $\abs{V(D)}=3$. Then
	$\dot z\ddot z$ is an  edge, and 
	the  two inner faces of $G'$ incident with $\dot z\ddot z$ are contained in $V(G'_C)\cap V(G'_{C'})$.  
	Since the intersection of any two inner faces of $G'_C$ has at most $2$ vertices, 
	we have  $|V(G'_C)\cap V(G'_{C'})|\ge   4$, a contradiction.
	
As $V(G'_C)\cap V(G'_{C'})= \ol Y_{C} \cup \ol Y_{C'}$, we conclude that $\abs{\ol Y_{C'}}=2$  and $\abs{V(C)}=5=\abs{V(C')}$.

	Let $Q\in \mathcal Q$ be the plane graph isomorphic to $G'_C$.
	By inspection of Figure~\ref{octahedron}, 
	$G'_{C'}$ is isomorphic to $Q$.
	We may assume that $G'_C=Q$ by relabelling vertices. 
	Let $u\mapsto u'$ be an isomorphism from $G'_C$ to $G'_{C'}$.
	Using uniqueness from Observation~\ref{obs1}\ref{obs1uniq}, $z=\dot z$, $z'= \ddot z$, $\ol Y_{C}=\{v_{4},v_{5}\}$ and $\ol Y_{C'}=\{v'_{4},v'_{5}\}$.
	To prove our claim let us divide our analysis into two cases, resulting either in $H_1$ or $H_2$.
	\begin{itemize}
		\item 	If $\abs{V(D)}=4$, then $Q=Q_4^{+}$ and 
		$V(G'_C)\cap V(G'_C)=V(D)=\{v_4,v_5,v_1,z\}$. Since $v_4,v_5\in  \ol Y_C$, we have $v_1,z\in \ol Y_{C'}$. Then $v_4'=v_1$, $v_5'=z$, and $v_5=z'$. 
		As $X_C$ and $X_{C'}$ are exposed in $G'$, 
		the cycle $v_1v_2v_3v_1'v_2'v_3'v_1$ is in $G'[B(G')]$ and so $H=H_1$ in Figure~\ref{fig(a)}.

		\item If $\abs{V(D)}=3$, then 
		$Q=Q_4^{++}$, $V(D)=\{z,v_4,v_5\}$.
		By symmetry, we may assume that $z'=v_5$. 
		Then $V(G'_C)\cap V(G'_{C'})=\{z, v_4,v_5,v_1\}$, as $C'$ contains all common neighbors of $z$ and $z'$ in $G'$, which is a property of $Q_4^{++}$.
		Since $v_4,v_5\in \ol Y_C$ and $V(G'_C)\cap V(G'_{C'})\subseteq \ol Y_C \cup \ol Y_{C'}$, we deduce that $v_1,z\in \ol Y_{C'}$.
		By symmetry in $G'_{C'}$, we may assume that $z=v_5'$ and $v_1=v_4'$. 
		As $X_{C}$ and $X_{C'}$ are exposed in $G'$, the cycle $v_{1}v_{2}v_{3}v'_{1}v'_{2}v'_{3}v_{1}$ is in $G'[B(G')]$.
		So $H=H_{1}+zz'=H_{2}$ in Figure~\ref{fig(b)}.
	\end{itemize}

	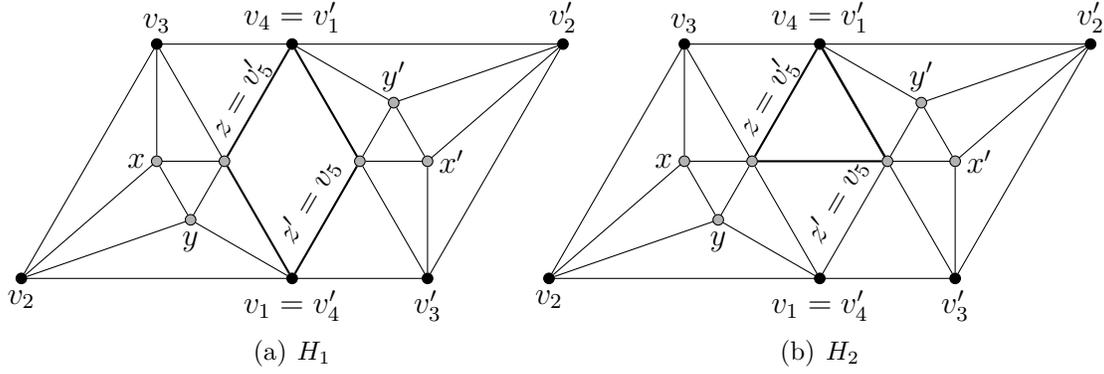
\begin{figure}
		\centering
		\subfigure[	$H_{1}$]{
			\label{fig(a)}
			\begin{tikzpicture}[scale=.9]
				\tikzstyle{every label}=[label distance=2pt,rectangle,fill=none,draw=none]
				\tikzstyle{every node}=[circle,draw,fill=black!30,inner sep=0pt,minimum width=4pt]
				\tikzstyle{xc}=[circle,draw,fill=black,inner sep=0pt,minimum width=4pt]
				\draw (0,0) node [xc,label=below:$v_2$] (v2) {}
				--++ (0:4) node [xc,label=below:{$v_1=v_4'$}] (v1){}
				--++ (0:2) node [xc,label=below:{$v_3'$}] (v3'){};
				\draw (v2)
				--++(60:4) node[xc,label=$v_3$](v3){}
				--++ (0:2) node[xc,label={$v_4=v_1'$}] (v4){}
				--++ (0:4) node [xc,label={$v_2'$}] (v2'){}
				--(v3');
				\draw (v3)--++(-60:2)node[label={[rotate=60]above right:{$~z=v_5'$}}](z){};
				\draw [thick](v4)--+(-60:2)node[thin,label={[rotate=60]above left:{$z'=v_5~$}}](z'){};
				\draw (z')--(v3');
				\draw (z')--+(60:1)node[label={$y'$}](y'){}--(v2');
				\draw (z')--+(0:1)node[label=right:{$x'$}](x'){}--(v2');
				\draw (v4)--(y')--(x')--(v3');
				\draw (z)--+(180+60:1)node[label=below:{$y$}](y){}--(v2);
				\draw (z)--+(180:1)node[label=left:{$x$}](x){}--(v2);
				\draw (v3)--(x)--(y)--(v1);
				\draw [thick](z')--(v1)--(z)--(v4);
			\end{tikzpicture}
		}\hskip-2em
		\subfigure[	$H_{2}$]{
			\label{fig(b)}
			\begin{tikzpicture}[scale=.9]
				\tikzstyle{every label}=[label distance=2pt,rectangle,fill=none,draw=none]
				\tikzstyle{every node}=[circle,draw,fill=black!30,inner sep=0pt,minimum width=4pt]
				\tikzstyle{xc}=[circle,draw,fill=black,inner sep=0pt,minimum width=4pt]
				\draw (0,0) node [xc,label=below:$v_2$] (v2) {}
				--++ (0:4) node [xc,label=below:{$v_1=v_4'$}] (v1){}
				--++ (0:2) node [xc,label=below:{$v_3'$}] (v3'){};
				\draw (v2)
				--++(60:4) node[xc,label=$v_3$](v3){}
				--++ (0:2) node[xc,label={$v_4=v_1'$}] (v4){}
				--++ (0:4) node [xc,label={$v_2'$}] (v2'){}
				--(v3');
				\draw (v3)--++(-60:2)node[label={[rotate=60]above right:{$~z=v_5'$}}](z){};
				\draw [thick](v4)--+(-60:2)node[thin,label={[rotate=60]above left:{$z'=v_5~$}}](z'){};
				\draw (z')--(v3');
				\draw (z')--+(60:1)node[label={$y'$}](y'){}--(v2');
				\draw (z')--+(0:1)node[label=right:{$x'$}](x'){}--(v2');
				\draw (v4)--(y')--(x')--(v3');
				\draw (z)--+(180+60:1)node[label=below:{$y$}](y){}--(v2);
				\draw (z)--+(180:1)node[label=left:{$x$}](x){}--(v2);
				\draw (v3)--(x)--(y)--(v1);
				\draw [thick](z')--(v4)--(z)--(z');
				\draw (z')--(v1)--(z);
			\end{tikzpicture}
		}
		\caption{The isomorphism types $H_{1}$ and $H_{2}$ of $H=G'[V(G'_{C})\cup V(G'_{C'})]$ when $|V(D)|=4$ or $|V(D)|=3$ in the proof of Lemma~\ref{lem-delta}. Solid black vertices denote boundary vertices of $G$ and thick edges represent edges in $D$.}
		\label{figA} %
	\end{figure}

	Notice that in both cases, $v_4=v_1'\in B(G')$ and $v_4\in V(D)$. 
	Set $Y'=\{v_{4}, x', y'\}$ and $G''=G-Y'$.  
	As $V(\inn_G(D))=Y$, in $G-v_4$, we can collect both $x'$ and $y'$ and at least one vertex of $Y$ is exposed. 
	Thus $B(G'')-B$ contains $z,z'$ and $(B(G'')-B)\cap Y\ne\emptyset$. So    $|B(G'')-B| \ge 3$.

	Let $\mathcal P$ be an optimal special cycle packing of $G''$, 
	and put \[\mathcal{P}_{0}=\{C^*\in\mathcal{P}: C^* \text{ is 
	a non-exposed special cycle of } G\}.\] Consider $C^{*}\in \mathcal{P}_{0}$.
	As $v_{4}=v_1'\in B\cap Y'$,  there is no exposed type-c special cycle in $G''$.
	Thus $C^*$ is type-b, and so $X_{C^*}\cap (B(G'')-B) \neq \emptyset$. 
	Let $w\in X_{C^*}\cap (B(G'')-B)$.
	Since $T_{C^*}$ is connected, has a neighbour of $w$, and has no vertex from $B(G'')$, we have
 $V(T_{C^*})\subseteq Y$
	and $X_{C^*}\subseteq  (B(G'')-B)\cup \{v_1\}$.

	As $\mathcal{P}_{0}$ is a packing, $3|\mathcal{P}_{0}|\le |B(G'')-B|+1$. 
	This implies that $|\mathcal {P}_{0}|\le|B(G'')-B|-2$, 
	because $|B(G'')-B|\ge 3$. 
	We now deduce that 
	\[|\mathcal{P}_{0}|\le |B(G'')-B|-2=|B(G'')|-(|B|-1)-2=|B(G'')|-|B|-1.\]
	Therefore 
	\[\tau(G) \ge \tau(G'')-|\mathcal{P}_{0}|\ge \tau(G'') - |B(G'')| + |B|+1.\]
	Hence, using $V(G)=V(G'')\cup Y'$,  
	\[\partial(G)=\frac34|V(G)|+\frac14(|B|-\tau(G))\le\frac34 (|V(G'')|+3)+\frac14(|B(G'')|-\tau(G'')-1) = \partial(G'')+ 2.\]
	Now, as we have already collected $x',y'$, we have 
	\[f(G;A) \ge f(G'';A)+2 \ge \partial(G'')+2 \ge \partial(G).\]
	This contradicts the assumption that $G$ is a counterexample.
\end{proof}

\begin{lemma}
\label{lem-taug'}
Let $Y$ be a nonempty subset of $V(G)$ such that $G[Y]$ is connected and let $G'=G-Y$. Then   
\[
	\partial(G) \le \partial(G') +\frac{ 3|Y|}{4} + \frac{|B-B(G')|+\delta(Y)}{4}. %
\]
Moreover, if $G$ has an exposed special cycle $C$ such that $Y_C \cap Y \ne \emptyset$ and $Y_C\cap Y_{C'}=\emptyset$  for any other exposed special cycle $C'$ of $G$, then  
\[
	\partial(G) \le \partial(G') + \frac{ 3|Y|}{4} + \frac{|B-B(G')|+\delta(Y)-1}{4}.%
\]
\end{lemma}
\begin{proof}
	
	An optimal special cycle packing of $G'$
	has at most $\abs{B(G')-B}$ type-b cycles by definition
	and 
	has $\delta(Y)$ type-c cycles by Lemma~\ref{lem-delta}. 
	We can remove such cycles from the special cycle packing of $G'$  to obtain a special cycle packing of $G$.
	So \[\tau(G)   \ge \tau(G') - |B(G')-B| - \delta(Y) =  \tau(G') - |B(G')|+|B| - |B-B(G')| - \delta(Y).\]
	Plugging this into the definition of $\partial(G)$, we obtain  \[\partial(G) \le \partial(G') +\frac{ 3|Y|}{4} + \frac{|B-B(G')|+\delta(Y)}{4}.\]

	If $G$ has an exposed special cycle $C$ such that $C$ is not a special cycle of $G'$ and $Y_C$ is disjoint from $Y_{C'}$ for any other exposed special cycle $C'$ of $G$, then we can add cycle $C$ to the 
	special cycle packing of $G$ obtained above. So 
	$$\tau(G) \ge \tau(G') - |B(G')-B| - \delta(Y)+1=  \tau(G') - |B(G')|+|B| - |B-B(G')| - \delta(Y)+1.$$
	Plugging this into the definition of $\partial(G)$, we obtain 
	\[\partial(G) \le \partial(G') +\frac{ 3|Y|}{4} + \frac{|B-B(G')|+\delta(Y)-1}{4}.
	\qedhere\]
\end{proof}

\begin{lemma}
	\label{d4}Every vertex $v\in V-A$ satisfies $d(v)\geq4$.
\end{lemma}
\begin{proof}
	Suppose that $d(v)\leq3$. Apply   Lemma~\ref{lem-taug'} with $Y=\{v\}$.  
	Let $G'=G-Y$.
	Note that if 
	$v$ is a boundary vertex, then $\delta(Y)=0$. So 
	$|B-B(G')|+\delta(Y) \le 1$.   Therefore  $$\partial(G) \le  \partial(G') + \frac34 + \frac14.$$ 
	 By the minimality of $(G;A)$,  $f(G'; A) \ge \partial(G')$. Therefore 
	 $f(G;A) = f(G';A)+1 \ge \partial(G)$, a contradiction. 
\end{proof}	

\begin{lemma}\label{lem-a}
	There are no disjoint nonempty subsets $X$, $Y$  of $V(G)$ such that
	$Y  $ is a set of $4|X|$ interior vertices of $G$,  
	$G[X\cup Y ]$ is connected, and 
	$Y$ is collectable in $G-X$.
\end{lemma}
\begin{proof}
    Suppose that there exist disjoint nonempty sets $X, Y\subseteq V(G)$ 
	such that $Y$ is a subset of $4\abs{X}$ interior vertices of $G$,
	$G[X\cup Y]$ is connected, 
	and $Y$ is collectable in $G-X$.
	Let $G'=G-(X\cup Y)$.
    We apply Lemma~\ref{lem-taug'}. Since $|B-B(G')|+ \delta(X\cup Y) \le |X|$, we have  
	  $\partial(G) \le \partial(G') +\frac{3}{4} (|X|+|Y|) + \frac{1}{4}|X|=\partial(G')+4|X|$. 
	  As $G$ is extreme, $f(G';A) \ge \partial(G')$. 
    Hence $f(G;A) \ge f(G';A)+|Y| =  f(G';A)+ 4|X| \ge \partial(G')+4|X|  \ge \partial(G)$, a contradiction. 
\end{proof}

 \begin{lemma}\label{two-special-cycle}  
	For any two distinct special cycles $C_1, C_2$ of $G$, 
	$Y_{C_1} \cap Y_{C_2} = \emptyset$. 
\end{lemma}
\begin{proof}
Assume to the contrary that $C_1, C_2$ are two   special cycles  of $G$   with $Y_{C_1} \cap Y_{C_2} \ne \emptyset$. 
	Observe that for each $i=1,2$, $V(T_{C_i})$ has two vertices of degree $4$
	and one vertex of degree $4$, $5$, or $6$ in $G$.

	If $T_{C_1}$ and $T_{C_2}$ share an edge, say $T_{C_1} = xyz$ and $T_{C_2} = xyz'$, then one of $x,y$, say $x$, has degree~$4$.
	Since $G$ is simple, $z\neq z'$.
	Let $v$ be the other neighbor of $x$. 
	By inspecting all graphs in $\mathcal Q$, we deduce that
	each of $z$, $z'$ is either adjacent to $v$ or has degree at most $5$ in $G$. So in $G-v$, the set $\{x,y,z,z'\}$ is collectable, contrary to Lemma~\ref{lem-a}.

	Assume $T_{C_1}$ and $T_{C_2}$ have a common vertex,   say   $T_{C_1} = xyz$ and $T_{C_2}=xy'z'$. 
	If none of $y$, $z$, $y'$, $z'$ have degree $6$, then we can delete $x$ and collect $y$, $z$, $y'$, $z'$, contrary to Lemma~\ref{lem-a}. 
	So we may assume that $d_G(y)=6$ and hence $d_G(x)=d_G(z)=4$ and all the faces incident to $x$ are triangles because $G_{C_1}$ is isomorphic to $Q_4^{++}$.
	Thus we may assume $yy', zz' \in E(G)$.
	By deleting $y$, we can collect $x$, $z$, $z'$, and $y'$, 
	again contrary to Lemma~\ref{lem-a}. (We collect $y'$ ahead of  $z'$ if $d_G(z')=6$ and collect $z'$ ahead of $y'$ otherwise.)
	Thus $V(T_{C_1}) \cap V(T_{C_2}) = \emptyset$. 
	
	If $X_{C_1} \cap V(T_{C_2}) \ne \emptyset$, then 
	for a vertex $v$ of maximum degree in $V(T_{C_2})$, 
	after deleting~$v$, we can collect the other two vertices of $T_{C_2}$ and two vertices of $T_{C_1}$, contrary to Lemma~\ref{lem-a}.
	So $X_{C_1}\cap V(T_{C_2})=\emptyset$ and by symmetry, 
	$X_{C_2}\cap V(T_{C_1})=\emptyset$.
	
	If  $  X_{C_1} \cap X_{C_2}$ contains a vertex $v$, then by deleting $v$, we can collect two vertices from each of $T_{C_1}$ and $T_{C_2}$, again contrary to Lemma~\ref{lem-a} because $V(T_{C_1})\cap V(T_{C_2})=\emptyset$. 
	\end{proof}

\begin{lemma}
\label{lem-typec}
If $C$ is a special cycle of $G$, then there is a vertex $u \in X_C$ such that $V(T_C)$ is collectable in $G-u$ and $G'=G-(V(T_C) \cup \{u\})$ has no type-c special cycle. 
\end{lemma}
\begin{proof}
	Suppose  the lemma fails for some special cycle $C$ of $G$ with $|E(C)|=k$.
	Then $G_C$ is isomorphic to a graph $Q\in \mathcal Q$.
	We may assume $G_{C}=Q$. 
	Then $V(T_{C})$ is collectable 
	in $G-v_1$.  Put $Y=V(T_C) \cup \{v_1\}$ and $G'=G-Y$. 
	Since $G'$ has a type-c special cycle $C'$, 
	$G'_{C'}$ has a facial cycle $C''$ with $Y= V(\inn_{G}(C''))$.  
	
	Then $C''$ consists of the subpath $C-v_1$ from $v_2$ to $v_k$ of length $k-2$
	and a path $P$ from $v_k$ to $v_2$ in $G'$.   As $G'_{C'}$ is
	special, $3\le |E(C'')|\le5$. So $|E(P)|\le5-(k-2)\le4$. 
	Now $N_G(v_{1})\subseteq V(P)\cup\{y,z\}$,
	so $d_{G}(v_{1})\le |E(P)|+3\le 10-k\le7$. 
	If $d_{G}(v_{1})\le6$, then after deleting $v_{2}$ we can collect $Y$: use the order $x,y,v_{1},z$ if $d_{G}(v_{1})\le 5$; else $d_{G}(v_{1})=6$ and $k\le4$, so use the order $x,y,z,v_{1}$. This contradicts
	Lemma~\ref{lem-a}. 
	Thus $d_{G}(v_{1})=7$. So $k=3$, $|E(P)|=4$, $|E(C'')|=5$, 
	$G_{C}=Q_{1}$, and $v_1$ is adjacent to all vertices of $P$.
	
	Setting $u=v_{3}$, and using symmetry between $v_1$ and $v_3$, 
	we see that $v_{3}$ is also an interior vertex with $d_{G}(v_{3})=7$. 
	
	Let $P=v_{3}u_{1}u_{2}u_{3}v_{2}$. 
	Now $G'_{C'}$ is isomorphic to $Q_{4}$ since $C''$ is a facial $5$-cycle. Assume $u\mapsto u'$ is
	an isomorphism from  $Q_{4}$ to $G'_{C'}$. 
	Then $C''=z'v'_{3}v'_{4}v'_{5}v'_{1}z'$.
	
	If there is
	$w\in \{v_{2},v_{3}\}\cap \{v'_{1},v'_{3}\}$ then after deleting $w$ 
	we can collect $\{x,y,z,x',y',z'\}$, contrary to  Lemma~\ref{lem-a}. Else 
	$\{v'_1,v'_3\}=\{u_1,u_3\}$ and therefore $z'=u_{2}$, see Figure~\ref{figC}.
	After deleting $\{v_{2},u_{1}\}$, we can collect $\{x,y,z,v_{3},v_{1},z',x',y'\}$, 
	contrary to Lemma~\ref{lem-a}, as both $v_1$ and $v_3$ are interior vertices.
\end{proof}

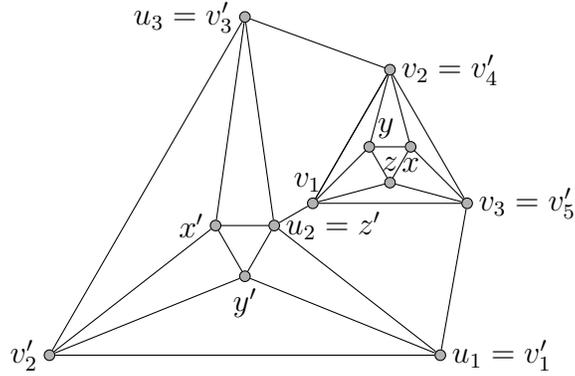
\begin{figure}
	\centering
\begin{tikzpicture}[scale=1]
	\tikzstyle{every label}=[label distance=2pt,rectangle,fill=none,draw=none]
	\tikzstyle{every node}=[circle,draw,fill=black!30,inner sep=0pt,minimum width=4pt]
	\node [label=left:{$u_3=v_3'$}] at (90:3) (x) {};
	\node [label=left:$v_2'$] at (210:3)(y) {};
	\node [label=right:{$u_1=v_1'$}] at (330:3)(z) {};
	\node [label=right:{$v_3=v_5'$}] at (10:3)(v) {};
	\node [label=right:{$v_2=v_4'$}] at (50:3) (u){};
	\node  [label=left:$x'$] at (150:.45) (v0) {};
	\node  [label=below:$y'$] at (150+120:.45) (v1) {};
	\node  [label={right:{$u_2=z'$}}] at (150+240:.45) (v2) {};
	\draw (u)--++ (-120:{6*sin(20)}) node [label=above:$v_1~$] (vv1){};
	\draw (vv1)--++(45:{3*sin(20)/cos(15)}) node [label=above right:$y$] (yy){}--(u);
	\draw (u)--++(-75:{3*sin(20)/cos(15)}) node [label=below:$x$] (xx){}--(v);
	\draw (v)--++(165:{3*sin(20)/cos(15)}) node [label=above:$z$] (zz){}--(vv1);
	\draw (xx)--(yy)--(zz)--(xx);
	\draw (v1)--(v2)--(v0)--(v1);
	\draw (v0)--(y)--(v1)--(z)--(v2)--(x)--(v0);
	\draw (x)--(y)--(z);
	\draw (x)--(u)--(v)--(z);
	\draw (v2)--(vv1)--(u) (vv1)--(v);
\end{tikzpicture}
\caption{The graph $\inn_{G}(C'')$ in the last part of the proof of Lemma~\ref{lem-typec} when $z'=u_2$.
Note that $d_G(v_3)=7$ and $v_3$ is an interior vertex.}\label{figC}
\end{figure}

	\begin{lemma}
	\label{lem-nospecialcycle}
	$G$ has no  special cycle.
	\end{lemma}
	\begin{proof}
	Assume to the contrary that $C$ is a   special cycle of $G$. 
	By Lemma~\ref{lem-typec}, there is a vertex   $u \in X_{C}$ such that 
	$V(T_C)$ is collectable in $G-u$ and $G'=G-(V(T_C)\cup\{u\})$ has no type-c special cycle. Observe that $f(G;A)\ge f(G';A)+3$. So it suffices to show that $\partial(G)\le \partial(G')+3$.
	Since $G'$ has no type-c special cycles, 
	every exposed special cycle of $G'$ is a special cycle of $G$. 
	
	If $u\notin B$, then $B(G')=B$ and so $B-B(G')=\emptyset$.
	As $\delta(V(T_C)\cup\{u\})=0$, we deduce from Lemma~\ref{lem-taug'} that
	$\partial(G)\le \partial(G')+\frac{3}{4}\cdot 4$. 
	
	Thus we may assume that $u\in B$ and so 
	$\abs{B-B(G')}=1$.
	If $C$ is exposed in $G$, then by Lemmas~\ref{lem-taug'} and \ref{two-special-cycle}, 
	$\partial(G)\le \partial(G')+\frac34\cdot 4+\frac{1-1}{4}$.
	
	If $C$ is not exposed in $G$, then $X_C$ has some interior vertex $v$.
	Since $v$ is adjacent to a vertex of $T_C$, $v$ is exposed in $G'$. 
	By Lemma~\ref{two-special-cycle}, 
	$v\notin X_{C'}$ for every exposed special cycle $C'$ of $G'$, 
	because $C'$ is a special cycle of $G$.
	Therefore, in an optimal special cycle packing of $G'$, 
	at most $\abs{B(G')-B}-1$ of the cycles are not exposed in $G$. 
	So, 
	\[\tau(G)\ge \tau(G')-(\abs{B(G')-B}-1)
	=\tau(G')-(\abs{B(G')}-\abs{B}+1)+1.\]	
	Thus
	\begin{align*}
		\partial(G)&=\frac34\abs{V(G)}+\frac14(|B|-\tau(G))\\
		&\le  \frac34(\abs{V(G')}+ 4)+\frac14(|B|+(-\tau(G')+\abs{B(G')}-\abs{B}))
		=\partial(G')+3.	\qedhere	
	\end{align*}
\end{proof}

\begin{lemma}
	\label{cor-red}
	Let $s$ be an integer.
	Let $X$ and $Y$ be disjoint subsets of $V(G)$ such that 
	$Y$ is collectable in $G-X$.
	If $|B(G-(X\cup Y))| \ge |B(G)|+s$, $G[X \cup Y]$ is connected, and $(X \cup Y) \cap B(G) \ne \emptyset$,
	then $s+|Y| < 3|X|$.
\end{lemma}
\begin{proof}
	Let $G'=G-(X\cup Y)$. 
	Since $(X \cup Y) \cap B(G) \ne \emptyset$ and $G[X \cup Y]$ is connected,
	any special cycle of $G'$ is also a special cycle of $G$. So $\tau(G')=0$
	and $\partial(G) \le \partial(G') + \frac34 (|X\cup Y|) - \frac{s}4$. 
	As $X\cup Y\neq \emptyset$ and $(G;A)$ is extreme, $f(G';A)\ge \partial(G')$.
	Thus
	\[f(G';A)+|Y|\le f(G;A)<\partial(G) \le \partial(G') + \frac34 |X\cup Y| - \frac{s}4\le f(G';A) + \frac34 |X\cup Y| - \frac{s}4.\]
	This implies that $s+|Y|< 3|X|$.
\end{proof}

\begin{lemma}
  \label{cutvertex}
  $G$ is $2$-connected and $|A|=2$.
\end{lemma}
\begin{proof}
	Suppose $G$ is not $2$-connected.
	If $\abs{V(G)}\le 3$, then
	$G$ is $(3,A)$-degenerate, so $f(G;A)=\partial(G)$ and we are done.
	Else $\abs{V(G)}> 3$. As $G$ is connected, it has a cut-vertex~$x$. 
	Let $G_1$, $G_2$ be subgraphs of $G$ such that $G=G_1 \cup G_2$,
	$V(G_1) \cap V(G_2)=\{x\}$, and $|V(G_{1})\cap A|\le |V(G_{2})\cap A|$.
	Observe that if $x\notin A$, then $A\cap V(G_1)=\emptyset$ by the choice of $G_1$ because $A$ is usable in $G$.

	Let $A_1=V(G_1)\cap A$ if $x\in A$ and $A_1=\{x\}$ otherwise.
	Let $A_2=V(G_2)\cap A$.
	Note that for each $i=1,2$,
	$A_i$ is usable in $G_i$.   
	For $i=1,2$, 
	let $X_i$ be a maximum $A_i$-good set in $G_i$.

	Let  $X:=(X_1 \cup X_2 - \{x\}) \cup (X_1 \cap X_2)$.   
	We claim that $X$ is $A$-good in~$G$.
	If $x\in A$, then collect $X_1-A$, $X_2-A$, $A\cap X$.
	If $x\notin A$ and $x\in X_1\cap X_2$, then collect $X_1-\{x\}$, $X_2$.
	If $x\notin A$ and $x\notin X_1\cap X_2$, then collect $X_1-\{x\}$, $X_2-\{x\}$.
	This proves the claim that $X$ is $A$-good in $G$.
	
As $(G;A)$ is extreme, $f(G_i;A_i) \ge \partial(G_i)$ for $i=1,2$.  

If $x \in B$ then $B(G_i)=B(G)\cap V(G_i)$ for $i=1,2$. 
Note that any special cycle of $G_i$ is a special cycle of $G$ and  
so $\tau(G_i)=0$ for $i=1,2$
by Lemma~\ref{lem-nospecialcycle} and hence
 $\partial(G) = \partial(G_1)+ \partial(G_2) - 1$. 

 If $x \notin B$, then we may assume $V(G_1)\cap B(G)=\emptyset$.
 Hence $B(G)=B(G_2)$.
 Since only one inner face of $G_2$ contains vertices of $G_1$,
 $\tau(G_2)\le 1$ by Lemma~\ref{lem-nospecialcycle}.
 Note that \[\partial(G)=\partial(G_1)+\partial(G_2)-\frac{3}{4}
 +\frac{1}{4}\tau(G_2)
 -\frac{1}{4}(\abs{B(G_1)}-\tau(G_1)).\]
 Since $\tau(G_1) \le |B(G_1)|-2$, we have 
 $\partial(G) \le \partial(G_1)+ \partial(G_2) - 1$.

In both cases, we have the contradiction:
\[
  f(G;A)%
  \ge\abs{X_1}+\abs{X_2}-1= f(G_1;A_1)+f(G_2;A_2)-1\ge \partial(G_1)+ \partial(G_2) - 1\ge \partial(G).
\]
Thus $G$ is $2$-connected, and hence   $|A| \le 2$. As $(G;A)$ is extreme,
 we have $|A|=2$.
\end{proof}

In the following, set  $A=\{a,a'\}$.
 
\begin{lemma}\label{nochord}
The boundary cycle $\mathbf{B}$ has no chord.
\end{lemma}
\begin{proof}
Assume $\mathbf{B}$ has a chord $e:=xy$. 
Let $P_1$, $P_2$ be the two paths from $x$ to $y$ in $\mathbf{B}$ such that $A\subseteq V(P_1)$.
Since $e$ is a chord, both $P_1$ and $P_2$ have length at least two.
 
Set $G_{1}=\inn[P_1+e]$
and $G_{2}=\inn[P_2+e]$. As $\tau(G)=0$ by Lemma~\ref{lem-nospecialcycle}, we know that $\tau(G_1)=\tau(G_2)=0$. Hence 
$\partial(G)=\partial(G_1)+\partial(G_2)-2$.
We may assume that $A \subseteq V(G_2)$. Let $A_1=\{x,y\}$ and $A_2=A$.

For $i=1,2$, let $X_i$ be a  maximum $A_i$-good set in $G_i$. 
Then $X = (X_1 \cup X_2 - \{x,y\}) \cup (X_1 \cap X_2)$ is an $A$-good set in $G$:
collect $X_1-\{x,y\}$, $(X_2-\{x,y\})\cup (X_1\cap X_2)$.
Thus
\[f(G;A) \ge f(G_1;A_1) + f(G_2;A_2) - 2\ge \partial(G_1)+\partial(G_2)-2= \partial(G),\]
contrary to the choice of $G$.
\end{proof}

\begin{lemma}
\label{lem-neartriangulation}
$G$ is a near plane triangulation.
\end{lemma}
\begin{proof}
	By Lemma~\ref{cutvertex}, every face boundary of $G$ is a cycle of $G$.
Assume to the contrary that $G$ has an interior face $F$ which is not a triangle.
Then $V(F)$ has a pair of vertices non-adjacent in $G$ because $G$ is a plane graph.
Let $e \notin E(G)$ be an edge drawn on $F$ joining them. Then  $G'=G+e$ is a plane graph with $B(G')=B(G)$.
As $G$ is extreme, $G'$ is not a counterexample. As $f(G';A) \le f(G;A)$, we conclude $\tau(G') > \tau(G)$, and hence $G'$ has an exposed special cycle $C$ and $e$ is an edge of $G'_C$. By \ref{obs1collect} of Observation~\ref{obs1},  there is a vertex $ v \in X_C$   such that after deleting $v$, we can collect  all the three vertices of $T_C$.  In $G-(V(T_C) \cup \{v\})$, all vertices in $(V_C \cup V(F)) - (V(T_C) \cup \{v\})$ are exposed. 
By  Lemma~\ref{cor-red}, none of these vertices can be an interior vertex of~$G$, because otherwise $|B(G-(V(T_C)\cup \{v\})| \ge |B(G)|$. So all these vertices are boundary vertices of $G$. 
By Lemmas~\ref{cutvertex} and~\ref{nochord}, $G$ is $2$-connected,
$\abs{A}=2$, and 
 $B(G)$ has no chord, so $G$ has no other vertices and $\inn(B(G))=T_C$, as $v \in X_C$ is also a boundary vertex of $G$.  
By the definition of usable sets, the two vertices in $A$ are adjacent.

By Lemma~\ref{d4}, 
$\|u,V(T_C)\|\ge 2$ for every vertex $u\in B(G)-A$, and $\|w,B(G)\|\ge 2$ for every vertex $w\in V(T_C)$.   
On the other hand, the number of vertices $u\in B(G)$ with $\|u,V(T_C)\|\ge 2$ is at most $3$. So $|B(G)| \le 3+\abs{A}= 5$. 

If $|B(G)|=3$, then $G$ is triangulated.
Suppose $|B(G)|=4$. If $\|u,V(T_C)\|\ge2$ for three vertices $u\in B(G)$, then $G$ is isomorphic to $Q_2$; else $G$ is isomorphic to $Q_3$. Both are contradictions. If $|B(G)|=5$, then $G$ is isomorphic to $Q_4$ or $Q_4^+$, again a contradiction.  
\end{proof}

\section{Properties of separating  cycles}\label{sec:separating}
	 
In a plane graph $G$, a cycle $C$ is called \emph{separating} if 
both $V(\inn(C))$ and $V(\ext(C))$ are nonempty.
In this section we will discuss properties of
separating cycles in $G$.

\begin{lemma}\label{triangle}
	Suppose $T$ is a separating triangle of $G$ and let $I=\inn(T)$. Then
	\begin{enum}[label=(\alph*)]
	  \item $\|V(T),V(I)\|\ge6$, 
	  \item $|I|\ge3$,
	  \item $\|x,V(I)\|\ge 1$ for all $x\in V(T)$,
	  and
	  \item for all distinct $x$, $y$ in $V(T)$,
	  $\abs{N(\{x,y\})\cap V(I)}\ge 2$.
	 \end{enum}
\end{lemma}
\begin{proof} If $|I| \le 2$, then $I$ contains a vertex $v$ with $d_G(v) \le 3$, contrary to Lemma~\ref{d4}.
	Thus $|I| \ge 3$ and (b) holds. 
	Moreover, 
 $I^+:=\inn [T]$ is triangulated and therefore
 $\|I^+\|=3|I^+|-6$ and $\|I\|\le 3|I|-6$.
 Thus 
$$\|V(T),V(I)\|=\|I^+\|-\|T\|-\|I\|\ge 3(3+|I|)-6-3-(3|I|-6)=6.$$  Thus (a)   holds. 
As $I^+$ is triangulated and $T$ is separating,  
  every edge of $T$ is contained in a triangle of $I^+$ other than $T$; so (c) holds.

If $\abs{(N(x)\cup N(y))\cap V(I)}\le 1$, then 
$\abs{I}=1$ because $G$ is a near plane triangulation. 
This contradicts (b). So (d) holds. 
\end{proof}

\begin{lemma}\label{cor-reducible}
Let $C$ be a separating cycle in $G$ such that $V(C)\cap A=\emptyset$.
Assume $X$, $Y$ are disjoint subsets of   $G$ such that 
$X \cup Y \ne \emptyset$,
$Y$ is collectable in $G-X$, and $G[ X \cup Y]$ is connected. 
Let $G_1=\inn[C]-(X\cup Y)$,   $G_2=\ext(C)-(X\cup Y)$,
$B_1=B(G_1)$, $B_2=B(G_2)$,  
$G_2'=\ext[C]-(X\cup Y)$, 
$A'=V(C)-(X\cup Y)$. 
If $A'$ is  usable in $G_1$ and   collectable in $G_2'$, then
$$  |Y| +  |B_1| + |B_2|   < 3|X|+|B|+ \tau(G_2) \le 3|X|+|B|+1.$$ 
In particular, 
\[
\abs{Y}< \begin{cases}
3\abs{X}+\abs{B}-\abs{B_1}-\abs{B_2}
&\text{if }(X\cup Y)\cap B\neq \emptyset,\\
3\abs{X}+\tau(G_2)-\abs{B_1}
&\text{otherwise.}
\end{cases}
\]
\end{lemma}
\begin{proof}
    Since $A'$ is usable, $(X\cup Y)\cap V(C)\neq\emptyset$ and 
    so $X \cup Y$ lies  in the infinite  face of $G_1$. 
    Thus any special cycle of $G_1$ is also a special cycle of $G$. Thus by Lemma~\ref{lem-nospecialcycle}, $\tau(G)= \tau(G_1)=0$. 
	By Lemma~\ref{lem-delta},
	in an optimal special cycle packing of $G_2$, 
	at most one cycle is type-c 
	and there are no type-a or type-b cycles.
	Therefore
	$\tau(G_2)\le 1$.

 As  $A'$ is collectable in $G'_2$,
we have  $$f(G;A) \ge f(G_1;A') + f(G_2;A) + |Y|.$$ 
On the other hand,
 $$\partial(G)=\partial(G_1)+\partial(G_2) + \frac34 (|X|+|Y|) - \frac14  (|B_1| + |B_2|-|B| - \tau(G_2)) .$$

 As  $f(G_1;A') \ge \partial(G_1)$ and $f(G_2;A) \ge \partial(G_2)$, we have   
 $$\partial(G) - \frac34 (|X|+|Y|) +  \frac14  (|B_1| + |B_2|-|B| - \tau(G_2)) \le 
  f(G_1;A')+f(G_2;A)  \le f(G;A) - |Y|.$$ 
  As $f(G;A) < \partial(G)$, it follows   that  
\[|Y| +  |B_1| + |B_2|   < 3|X|+|B|+ \tau(G_2) \le 3|X|+|B|+1.
   \]

Note that if $(X \cup Y)\cap B\neq\emptyset$, then $\tau(G_2)=0$. In this case, we have  
  $$  |Y| +  |B_1| + |B_2|   < 3|X|+|B|.$$
  If $(X \cup Y)\cap B=\emptyset$, then $B_2=B$. In this case, we have
  $ |Y| +  |B_1|     < 3|X| + \tau(G_2)$.
\end{proof}

\begin{lemma}
	\label{sep-triangle}
	Let $C$ be a separating triangle of $G$. 
	If $C$ has no vertex in $B(G)$, then  either 
	$\|v,V(\ext(C))\| \ge 3$ for all vertices $v \in V(C)$ or  
	$\|v,V(\ext(C))\| \ge 4$ for two vertices $v\in V(C)$.
\end{lemma}
\begin{proof}
Suppose not. Let $C=xyzx$ be a counterexample with the minimal area. We may assume that 
$\|x, V(\ext(C))\| \le 2$ and  $\|y,V(\ext(C))\| \le 3$.
  By Lemma~\ref{triangle}(c), $z$ has 
  a neighbour $w$ in $I:=\inn(C)$.  
  If $w$ is the only neighbour of $z$ in $I$, then by Lemma~\ref{triangle}(b), 
  $C':=xwyx$ is a separating triangle. However, $w$ has only $1$ neighbour in $\ext(C')$
  and $x$ has at most $3$ neighbours in $\ext(C')$, contradicting the choice of $C$.
  
  Thus $\|z,V(I)\|\ge 2$.

  We apply Lemma~\ref{cor-reducible} with $C$, $X=\{z\}$ and $Y=\emptyset$. 
   Then $A':=\{x,y\}$ is usable in $G_1:=\inn[C]-z$, $A'$ is collectable in $G_2':=\ext[C]-z$ and $B_1:=B(G_1)\supseteq \{x,y\}\cupdot N_I(z)$.  So $|B_1|\ge4$, and this contradicts Lemma~\ref{cor-reducible}.
   \end{proof}
   
   \begin{lemma}\label{C4septri}
   Let $C$ be a separating induced cycle of length $4$ in $G$
   having no vertex in $B(G)$.
   Then exactly one of the following holds. 
   \begin{enum}
    \item $\abs{B(\inn(C))}\ge 4$.
    \item $\abs{V(\inn(C))}\le 2$ and every vertex in $\inn(C)$ has degree $4$ in $G$.
   \end{enum}
   \end{lemma}
   \begin{proof}
   Suppose that $|B(\inn(C))|\le 3$. 
   By Euler's formula, we have \[\|\inn [C]\|= 3\abs{V(\inn[C])}-7
   =3\abs{V(\inn(C))}+5\]
   as $G$ is a near plane triangulation. 
   Then since $C$ is induced, by Lemma~\ref{d4},
   \begin{equation}\label{eq:euler}
   \begin{split}
       0&\le \sum_{v\in V(\inn(C))} (d(v)-4)\\
	   &=\|\inn[C]\| - \|C\| + \| \inn(C)\|- 4|V(\inn(C))|\\
       &=(3|V(\inn(C))|+5)-4 + \| \inn(C) \|  - 4|V(\inn(C))|\\
       &=1-|V(\inn(C))|+\| \inn(C)\|  .
       \end{split}
   \end{equation}

   Suppose that  
   $\inn(C)$ has a cycle. Since $\abs{B(\inn(C))}\le 3$, 
   we deduce that $B(\inn(C)) = xyzx$ is a triangle. 
   By Euler's formula applied on $G[V(C)\cup B(\inn(C))]$, we have 
   \[\|V(C),B(\inn(C))\|= (3\cdot 7-7)-3-4=7,\] 
   hence $\mathbf B(\inn(C))$ is a facial triangle by Lemma~\ref{sep-triangle}. Therefore, $x,y,z$ have degree $4$, $4$, $5$ in $G$ by \eqref{eq:euler} and Lemma~\ref{d4}. 
   Let $w, w'\in V(C)$ be consecutive neighbours of $x$ in $V(C)$. From $G$, we can delete $w$ and collect $x,y,z$. Let $G'=G-\{w,x,y,z\}$. 
   If $G'$ has an exposed special cycle, then 
   the face of $G'$ containing $w$ has length at most $5$, implying that 
   $\| w, V(\ext(C))\| \le 2$ because $C-w$ is a subpath of an exposed special cycle of $G'$, as $C$ is induced. Then 
   we can delete $w'$ and collect $x,y,z,w$, contradicting Lemma~\ref{lem-a}.
   Therefore $G'$ has no exposed special cycles.
   Then $\partial(G)=\partial(G')+3$
   and $f(G;A) \ge f(G';A)+3 \ge \partial(G')+3 =\partial(G)$, a contradiction.

   Therefore $\inn(C)$ has no cycles.  
   Then 
    $\| \inn(C)\|\le \abs{V(\inn(C))}-1$,
    and so in \eqref{eq:euler} the equality 
    must hold. This means
    $\inn(C)$ is a tree and every vertex in $\inn(C)$ has degree $4$ in $G$ 
    by Lemma~\ref{d4}.
    If $\inn(C)$ has at least $3$ vertices,
    then let $w$ be a   vertex in $V(C)$ adjacent to some vertex in $\inn(C)$. 
    By deleting $w$, we can collect all the vertices in $\inn(C)$. Similarly we can choose 
    $w$ so that $G'=G-w-V(\inn(C))$ contains no special cycle, and that leads to the same contradiction.
      Thus we deduce (b).
   \end{proof}

\section{Degrees of boundary vertices} 	\label{sec:boundary}

\begin{lemma}
\label{deg<6}
Each vertex in $B$ has degree at most $5$.
\end{lemma}
\begin{proof}
Assume to the contrary that $x \in B$ has $d(x) \ge 6$. 
Then deleting $x$ exposes at least $4$ interior vertices.
Apply Lemma~\ref{cor-red} with $X=\{x\}, Y=\emptyset$ and $s=3$, we obtain a contradiction.    
\end{proof}

Recall that $A=\{a,a'\}$.

\begin{lemma}
\label{d=5}
Each vertex in $B-A$ has degree $5$.
\end{lemma}
\begin{proof}
Suppose that there is a vertex $x \in B-A$  with $d(x)<5$. By Lemma~\ref{d4}, $d(x)=4$. 
By Lemma~\ref{nochord}, exactly two of the neighbors of $x$ are in $B$.
Consider two cases.
		
\medskip
\noindent
{\it Case 1:}  $x$ has a neighbour $y \in B-A$. As    $|A|= 2$, we have $|B|\ge 4$.
As $G$ is a near plane triangulation,  there is a vertex $z\in N(x)\cap N(y)$ such that $xyzx$ is a facial triangle. As $\mathbf B$ has no chords by  Lemma~\ref{nochord},  $(N(x)\cap N(y))\cap B(G)=\emptyset$.

Suppose there is $z'\in N(x)\cap N(y)-\{z\}$.
Since $d(x)=4$ and $G$ is a near plane triangulation, $xzz'x$ is a facial triangle. Since $d(z)\ge 4$ by Lemma~\ref{d4}, $T:=yzz'y$ is a separating triangle.
As $d(y)\le 5$ by Lemma~\ref{deg<6}, $y$ has a unique neighbour $y'\in V(\inn (T))$ and therefore both $yy'zy$ and $yy'z'y$ are facial triangles.
By Lemma~\ref{triangle}(b), $\inn(T)$ contains at least three vertices and so $T':=zz'y'z$ is a separating triangle with $\|z,V(\ext(T'))\| = 2$ and $\|y',V(\ext(T'))\| = 1$, contrary to Lemma~\ref{sep-triangle}. 
So $N(x)\cap N(y)=\{z\}$.

If $d(y)=5$, then deleting $z$ and collecting $x$ and $y$ exposes  three vertices in $(N^{\circ}(x)\cup N^{\circ}(y))-\{z\}$, the resulting graph $G'=G-\{x,y,z\}$ has   $|B(G')|\ge|B|+1$. 
Apply Lemma~\ref{cor-red} with $X=\{z\}, Y=\{x,y\}$, and $s=1$, we obtain a contradiction.  

Hence $d(y)=4$. By repeating the same argument, we deduce that for all edges $vv'\in \mathbf B-A$, we have
\begin{enum*}
\item $d(v)=4=d(v')$ and
\item $|N(v)\cap N(v')|=1$.
\end{enum*}

Let $x'$, $y'$ be vertices such that  $N^{\circ}(x)=\{x',z\}$ and $N^{\circ}(y)=\{y',z\}$.  
As $G$ is a near plane triangulation and $\mathbf B$ is chordless, $G-B$ is connected. 
Let $J=\{x',z,y'\}$.
If $V-B\neq J$, then there exist $b\in J$ and $t\in (V-B)-J$ such that $b$ and $t$ are adjacent. 
Then 
deleting $b$ and collecting $x$, $y$ exposes 
all vertices in $(J-\{b\})\cup \{t\}$. 
Let $G'=G-\{x,y,b\}$. Then 
  $|B(G')|\ge|B|+1$. 
  With $X=\{b\}$, $Y=\{x,y\}$, and $s=1$, this contradicts Lemma~\ref{cor-red}.
   Hence $V-B=J$.

 \begin{figure}
	\centering
	\begin{tikzpicture}[yscale=.7]
	\tikzstyle{every node}=[circle,draw,fill=black!50,inner sep=0pt,minimum width=4pt]
	\node [label=$z$] at (0,.75)(z) {};
	\node [label=below:$x'$] at (-.6,.1) (x') {};
	\node [label=$x$] at (-1.3,1.2) (x) {};
	\node [label=left:$u$] at (-2,0) (u) {};
	\node [label=right:$v$] at (2,0) (v) {};
	\node [label=$y$] at (1.3,1.2) (y) {};
	\node [label=below:$y'$] at (.6,.1) (y') {};
	\draw (u)--(x');
	\draw (y')--(v);
	\draw (x)--(x');
	\draw (x)--(z)--(x');
	\draw (y)--(z)--(y');
	\draw (y)--(y');
	\scoped [on background layer] \draw (u) to [out=90,in=180] (0,1.5) to [out=0,in=90] (v);
	\scoped [on background layer] \fill [black,fill opacity=.2] (x'.center)--(u.center)--(u.south) to [out=-90,in=180] (0,-1.5) to [out=0,in=-90](v.south)-- (v.center)--(y'.center)--(z.center)--cycle;
	\draw [dashed] (u) to [out=-90,in=180] (0,-1.5) to [out=0,in=-90] (v);
	\end{tikzpicture}
	\caption{Case 1 in the proof of Lemma~\ref{d=5}. The dashed line may have other vertices and the gray region has other edges but no interior vertices.}
	\label{fig:d=5c1} %
\end{figure}
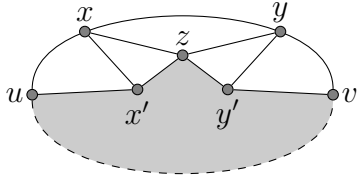

Let $u$, $v$ be vertices in $B$ so that $uxyv$ is a path in $\mathbf B$. Since $G$ is a near plane triangulation, $x'$  is adjacent to $u$ and $z$, and $y'$ is adjacent to $v$ and $z$, see Figure~\ref{fig:d=5c1}. Then $ux'zy,xzy'v$ are paths in $G$.
If $A=\{u,v\}$, then $\mathbf B$ is a $4$-cycle and as $d(x'), d(y') \ge 4$, we  must have $x'y' \in E(G)$, which implies that $G$ is isomorphic to $Q_2^+$ and $\mathbf B$ is a special cycle, contrary to Lemma~\ref{lem-nospecialcycle}.  
Therefore $A \ne \{u,v\}$ and 
since $y\notin A$, we deduce that $v\notin A$. This implies $d(v)=4$.  
Then $v$ has another neighbour in $J$, and 
by the observation that $y$ and $v$ have only one common neighbour $y'$, 
we deduce that $v$ is non-adjacent to $z$. 
Thus $v$ is adjacent to $x'$, and $x'$ is adjacent to $y'$. 

Furthermore every vertex in $B-\{u,x,y,v\}$ has degree at most $3$, 
because $\mathbf B$ has no chords and $x'$ is the only possible interior neighbor. By Lemma~\ref{d4}, every vertex in $B-\{u,x,y,v\}$ is in $A$. Then $G$ is isomorphic to $Q_4^{++}$ and $\mathbf B$ is a special cycle, contrary to Lemma~\ref{lem-nospecialcycle}.

\begin{figure}
	\centering
	\begin{tikzpicture}[xscale=.7,yscale=.3,rotate=-90]
        \tikzstyle{every node}=[circle,draw,fill=black!50,inner sep=0pt,minimum width=4pt]
        \filldraw [fill=black!20] 
           (60:3) node (a) [label=right:$a'\in A$]{}
           -- (120:.7) node (y)[label=below:$z$] {} 
           -- (240:.7) node (z)[label=below:$y$] {}
           -- (-60:3) node (b) [label=left:$a\in A$]{}
           -- cycle;
        \node at (180:3) (x)[label=$x$]{};
        \draw (a)--(x)--(b);
        \draw (x)--(y); \draw (z)--(x);
	\end{tikzpicture}
	\caption{Case 2 in the proof of Lemma~\ref{d=5}. The gray region may have other vertices.}\label{fig:d=5c2}
\end{figure}
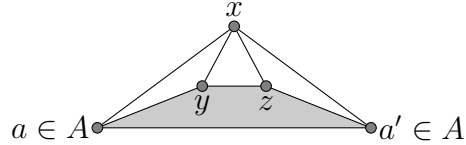

\medskip
\noindent
{\it Case 2:} $N_G(x) \cap B \subseteq A$.	 Then 
 $\mathbf{B}=xaa'x$. 
 Since $G$ is a near plane triangulation and $d(x)=4$,
 the neighbours of $x$ form a path of length $3$ from $a$ to $a'$, say $ayza'$ where $a$, $y$, $z$, $a'$ are the neighbours of $x$. (See Figure~\ref{fig:d=5c2}.)
 
If $|N^\circ(y)|\ge 3$, then deleting $y$ and collecting $x$ exposes at least three vertices in $N^{\circ}(y)$. 
Let $G'=G-\{x,y\}$. Then  $\abs{B(G')}\ge\abs{B}+2$. With $X=\{y\}, Y=\{x\}$, and $s=2$, this contradicts Lemma~\ref{cor-red}.  
 
Thus $|N^\circ(y)|\le 2$ and so $d(y) \le 5$. (Note that $y$ may be adjacent to $a'$.) By symmetry, $|N^\circ(z)|\le 2$ and $d(z)\le 5$.

If $y$ is adjacent to $a'$, then 
$z$ is non-adjacent to $a$ and so $d(z)=4$ by Lemma~\ref{d4}.
Then $T:=yza'y$ is a separating triangle, as $\inn(T)$ contains a neighbour of $z$. 
Since $d(y)\le 5$ and $d(z)=4$, we have $|N(\{y,z\})\cap V(\inn(T))|=1$, %
contrary to Lemma~\ref{triangle}(d).
 
So $y$ is non-adjacent to $a'$. 
By symmetry, $z$ is non-adjacent to $a$. As $\abs{N^{\circ}(y)}, \abs{N^{\circ}(z)}\le 2$ and $d(y), d(z)\ge 4$,  $y$ and $z$ have a unique common neighbour $w$ and $d(y)=d(z)=4$. Since $G$ is a near plane triangulation, 
$w$ is adjacent to both $a$ and $a'$.

If $d(w)>4$, then deleting $w$ and collecting $y$, $z$, $x$ exposes at least one vertex and so $|B(G-\{x,y,z,w\})|\ge |B|$. With $X= \{w\}, Y=\{x,y,z\}$, and $s=0$, this contradicts Lemma~\ref{cor-red}. 
This implies $d(w)=4$, hence $B(G)$ is a special cycle, contrary to Lemma~\ref{lem-nospecialcycle}.
\end{proof}
	
\section{The boundary is a triangle}\label{sec:triangle}
In this section we prove that $|B|=3$. 

\begin{lemma} If $xy\in E(\mathbf B-A)$,
	\label{property}
	then the following hold:
	\begin{enum}
	\item\label{su} There are  $S:=\{x_1, x_2, u, y_1, y_2\}\sub V-B$ and $x^*,y^*\in B$ such that   $x^*x_1 x_2 u y$ is a path in $G[N(x)]$ and 
	$x u y_1 y_2y^*$ is a path in $G[N(y)]$.
	\item \label{xuy} $d(x_2), d(u), d(y_1) \ge 5$.
	\item \label{dist}The vertices $x_1, x_2, u, y_1, y_2$ are all distinct.
		\item \label{7-4} $|N^{\circ}(\{x_2, u\})-S| \le 2$ and   $|N^{\circ}(\{y_1, u\})  -S | \le 2$.
		\item\label{indu-}   $x_2y_1,x_2y_2,x_1y_1,ux_1,uy_2\notin E$.
		\item \label{7-7} There is $w_1\in (N(\{x_2,u,y_1\})\cap B)-\{x,y\}$; in particular  $G[S]$ is an induced path.  
		\item\label{7-6} $x_2,u\notin N(x^*)$ and $y_1, u\notin N(y^*)$.    
		\item \label{7-8} Neither $x^*$ nor $y^*$ is equal to 
		the vertex $w_1$ from \ref{7-7}.
	\end{enum}
\end{lemma}
\begin{proof}
\ref{su} By Lemma~\ref{d=5}, $d(x)=5=d(y)$. By Lemmas~\ref{cutvertex} and \ref{nochord}, there are $x^*,y^*\in B$ with $N(x)\cap B=\{x^*,y\}$ and $N(y)\cap B=\{x,y^*\}$. As $G$ is a near plane triangulation, there is $u\in N(x)\cap N(y)$. So \ref{su} holds.

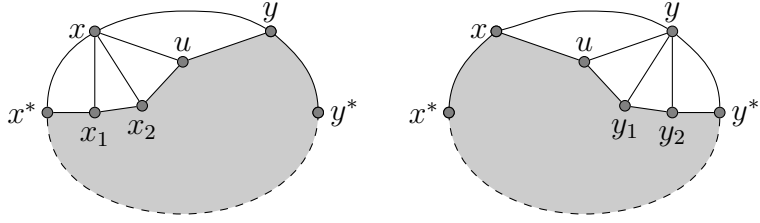
\begin{figure}
	\centering
	\begin{tikzpicture}[scale=.9]
	\tikzstyle{every node}=[circle,draw,fill=black!50,inner sep=0pt,minimum width=4pt]
	\node [label=$u$] at (0,.75)(u) {};
	\node [label=below:$x_2$] at (-.6,.1) (x2) {};
	\node [label=below:$x_1$] at (-1.3,0) (x1) {};
	\node [label=left:$x$] at (-1.3,1.2) (x) {};
	\node [label=left:$x^*$] at (-2,0) (x^*) {};
	\node [label=right:$y^*$] at (2,0) (y^*) {};
	\node [label=above:$y$] at (1.3,1.2) (y) {};
	\draw (x^*)--(x1)--(x)--(x2)--(u)--(y);
	\draw (x1)--(x2);
	\draw (x)--(u);
	\scoped [on background layer] \draw (x^*) to [out=90,in=180] (0,1.5) to [out=0,in=150] (y) ;
	\scoped [on background layer] \fill [black,fill opacity=.2] (u.center)--(x2.center)--(x1.center)--(x^*.center)--(x^*.south) to [out=-90,in=180] (0,-1.5) to [out=0,in=-90] (y^*.center) to [out=90,in=-40](y.center)--(u.center)--cycle;
	\draw [dashed,on background layer] (x^*) to [out=-90,in=180] (0,-1.5) to [out=0,in=-90] (y^*.center) ;
	\draw [on background layer] (y^*.center) to [out=90,in=-40] (y.center);
	\node  at (1.3,1.2)  {};
	\node  at (2,0)  {};
	\end{tikzpicture} 
	\quad
	\begin{tikzpicture}[scale=.9]
		\tikzstyle{every node}=[circle,draw,fill=black!50,inner sep=0pt,minimum width=4pt]
		\node [label=$u$] at (0,.75)(u) {};
		\node [label=left:$x$] at (-1.3,1.2) (x) {};
		\node [label=left:$x^*$] at (-2,0) (x^*) {};
		\node [label=right:$y^*$] at (2,0) (y^*) {};
		\node [label=above:$y$] at (1.3,1.2) (y) {};
		\node [label=below:$y_2$] at (1.3,0) (y2) {};
		\node [label=below:$y_1$] at (.6,.1) (y1) {};
		\draw (y^*)--(y2)--(y1)--(u)--(x);
		\draw (y)--(y1);
		\draw (y)--(u);
		\draw (y)--(y2);
		\scoped [on background layer] \draw (x) to [out=15,in=180] (0,1.5) to [out=0,in=150] (y);
		\scoped [on background layer] \fill [black,fill opacity=.2] (u.center)--(y1.center)--(y2.center)--(y^*.center)--(y^*.south) to [out=-90,in=0] (0,-1.5) to [out=180,in=-90] (x^*.center) to [out=90,in=-140](x.center)--(u.center)--cycle;
		\draw [dashed,on background layer](y^*.south) to [out=-90,in=0] (0,-1.5) to [out=180,in=-90] (x^*.center);
		\draw [on background layer] (x^*.center) to [out=90,in=-140] (x.center);
		\draw [on background layer] (y^*) to [out=90,in=-40] (y);
		\node at (-1.3,1.2)  {};
		\node at (-2,0)  {};
		\end{tikzpicture} 
	\caption{The situation in the proof of Lemma~\ref{property}\ref{xuy}.}
	\label{fig:x2deg} %
\end{figure}

\medskip\noindent
\ref{xuy} 
(See Figure~\ref{fig:x2deg}.)
As $d(u)\ge 4$ by Lemma~\ref{d4}, $x_2\ne y_1$.
Assume $d(x_2)=4$. 
If $x_2$ is adjacent to $y$, then 
$x_2=y_2$, implying that $d(x_2)>4$, contradicting the assumption. 
Thus $x_2$ is non-adjacent to $y$ and 
deleting $u$ and collecting $x_2$, $x$, $y$ exposes $y_1$, $y_2$ (note that it is possible that $x_1\in \{y_1,y_2\}$, so we do not count it as exposed).
We have $|B(G-\{u,x_2,x,y\})|\ge |B|$.
With $X=\{u\}, Y=\{x_2,x,y\}$, and $s=0$, this contradicts  
 Lemma~\ref{cor-red}. Thus $d(x_2)\ge 5$ by Lemma~\ref{d4}. 
 By symmetry, $d(y_1)\ge 5$. 
If $d(u)=4$, then we can delete $x_2$, collect $u,x,y$, and expose $y_1,y_2$.
This contradicts Lemma~\ref{cor-red} applied with $X=\{x_2\}, Y=\{u,x,y\}$, and $s=0$.
  So \ref{xuy} holds.

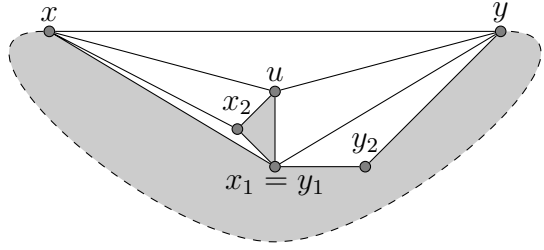
\begin{figure}
	\centering
        \begin{tikzpicture}
          \tikzstyle{every node}=[circle,draw,fill=black!50,inner sep=0pt,minimum width=4pt]
          \fill [fill=black!20]  
          (0,.5) node (u) [label=$u$] {}
          -- (0,-.5) node (v) [label={[yshift=2.8ex]below:{$x_1=y_1$}}] {}
          -- (-.5,0) node (x2) [label=above:$x_2$]   {}
          -- cycle;
          \node [label=above:$x$] at (-3,1.3) (x) {};
          \node [label=above:$y$] at (3,1.3) (y) {};
          \node [label=above:$y_2$] at (1.2,-.5) (y2) {};
          \scoped[on background layer]\fill [fill=black,fill opacity=.2,text opacity=1] 
          (x.center)--(x.west)
          to [out=-180,in=-180] (0,-1.5) 
          to [out=0,in=0] 
          (y.east)--(y.center)
          -- (y2.center)--(v.center) --cycle;
          \draw (x)--(u)--(y)--(y2)--(v)--(x)--(x2)--(u)--(v)--(x2);
          \draw (v)--(y);
          \draw (x) -- (y);
          \draw [dashed] (x) to [out=-180,in=-180] (0,-1.5) to [out=0,in=0] (y);
        \end{tikzpicture}
	\caption{When $x_1=y_1$ in the proof of  Lemma~\ref{property}\ref{dist}. Gray regions may have other vertices.}
	\label{B1-|N(xy)|=4-0} %
\end{figure}

\medskip\noindent
\ref{dist} Since $d(u)\ge 5$, we deduce $x_2\ne y_1$, and if $x_1=y_1$,  then  $T:=x_1x_2ux_1$ is a separating triangle (see Figure~\ref{B1-|N(xy)|=4-0}), since $d(x_2)\ge 5$. As $\|x_2,V(\ext(T))\|=1$ and $\|u,V(\ext(T))\|=2$, this contradicts Lemma~\ref{sep-triangle}.   So $x_1 \ne y_1$. By symmetry,    $x_2 \ne y_2$.

It remains to show that $x_1\ne y_2$. Suppose not. 
By \ref{xuy}, $d(x_2) 
\ge 5$, so $C:=x_1x_2uy_1x_1$ is a separating $4$-cycle (see Figure~\ref{B1-|N(xy)|=4-1}). We first prove the following.
\begin{equation}\label{cl1}
\text{For all $u'\in V(C)-\{u\}$, $|N(\{u,u'\})  \cap V(\inn(C))| \le 3$.}
\end{equation}
 Suppose not. 
 Then deleting $u$, $u'$ and collecting $x$, $y$ exposes
 two vertices in $V(C)-\{u,u'\}$ and at least $4$ vertices in $\inn(C)$.
 So 
 $|B(G-\{u,u',x,y\})|\ge|B|-2+2+4$.
This contradicts Lemma~\ref{cor-red} with $X=\{u,u'\}$, $Y=\{x,y\}$, and $s=4$. This proves \eqref{cl1}.

\begin{figure}
	\centering
	\begin{tikzpicture}
	\tikzstyle{every node}=[circle,draw,fill=black!50,inner sep=0pt,minimum width=4pt]
	\node [label=$u$] at (0,.75)(u) {};
	\node [label={[yshift=2.8ex]below:{$x_1=y_2$}}] at (0,-.75)(v) {};
	\node [label=right:$x_2$] at (-1,0) (x2) {};
	\node [label=left:$x$] at (-2,0) (x) {};
	\node [label=right:$y$] at (2,0) (y) {};
  \node [label=left:$y_1$] at (1,0) (y1) {};
  \scoped[on background layer]
  \fill [black,fill opacity=.2] 
    (u.center)--(x2.center)--(v.center)--(y1.center)--cycle
  	(x.center) to [out=-90,in=180] (0,-1.5) to [out=0,in=-90] (y.center)
    -- (v.center)  --cycle;	
  \scoped[on background layer]\draw [dashed] (x.center) to [out=-90,in=180] (0,-1.5) to [out=0,in=-90] (y.center);	
  \draw (x)--(u)--(y)--(y1)--(v)--(x)--(x2)--(u)--(y1)--(v)--(x2);
  \draw (v)--(y);
  \draw (x) to [out=90,in=180] (0,1.5) to [out=0,in=90] (y);  
	\end{tikzpicture}
	\caption{When $x_1=y_2$ in Lemma~\ref{property}\ref{dist}. Gray regions may have other vertices.}
	\label{B1-|N(xy)|=4-1} %
\end{figure}
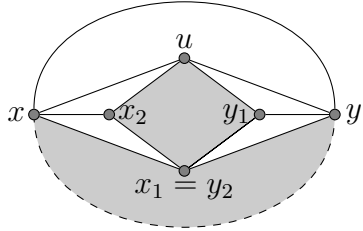

If $u$ is adjacent to $x_1$, then 
 $C_1:=x_1x_2ux_1$ and $C_2:=x_1uy_1x_1$ are both separating triangles by \ref{xuy}.  Then $|N(\{u,x_1\})\cap V(\inn(C_i))|\ge2$ for each $i\in \{1,2\}$ by Lemma~\ref{triangle}(d). Thus $|N(\{u,x_1\})  \cap V(\inn(C))| \ge 4$, contrary to \eqref{cl1}. So $u$ is non-adjacent to $x_1$.

If $x_2$ is adjacent to $y_1$, then $C_3:=ux_2y_1u$ is a separating triangle by \ref{xuy}. Then  $|N(\{u,x_2\})\cap V(\inn(C_3))|\ge2$ by Lemma~\ref{triangle}(d). As $\|u,V(\ext(C_3))\|=2$, Lemma~\ref{sep-triangle} implies that $\|x_2,V(\ext(C_3))\|\ge4$, hence $|N(\{u,x_2\})\cap V(\inn(x_1x_2y_1x_1))|\ge2$. Thus $|N(\{u,x_2\})  \cap V(\inn(C))| \ge 4$,  contrary to \eqref{cl1}.
So $C$ has no chord.

By \ref{xuy}, $C$ is a separating induced cycle of length $4$ in $G$. By Lemma~\ref{C4septri}, either $|B(\inn(C))|\ge 4$ or $\abs{V(\inn(C))}\le 2$ and every vertex in $\inn(C)$ has degree $4$ in~$G$.

By \eqref{cl1}, $d(x_2),d(y_1) \le6$.  
If $|B(\inn(C))|\ge 4$, then 
deleting $u$, $x_1$ and collecting $x$, $y$, $x_2$, $y_1$ exposes at least $4$ vertices 
and therefore $|B(G-\{u,x_1,x,y,x_2,y_1\})|\ge |B|+2$. 
This contradicts Lemma~\ref{cor-red} applied with $X=\{u,x_1\}$,  $Y=\{x,y,x_2,y_1\}$, and $s=2$.

Therefore we may assume $1\le \abs{V(\inn(C))} \le 2$
and every vertex in $\inn(C)$ has degree $4$ in $G$. 
As $x_2$ is non-adjacent to $y_1$, $x_1$ has at least one neighbour in $\inn(C)$
and therefore after deleting $x_1$, we can collect all vertices in $V(\inn(C))$ 
and then collect $x_2$, $y_1$ and $u$, this contradicts Lemma~\ref{lem-a}.
So \ref{dist} holds.

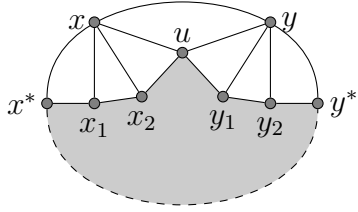
\begin{figure}
	\centering
	\begin{tikzpicture}[scale=.9]
	\tikzstyle{every node}=[circle,draw,fill=black!50,inner sep=0pt,minimum width=4pt]
	\node [label=$u$] at (0,.75)(u) {};
	\node [label=below:$x_2$] at (-.6,.1) (x2) {};
	\node [label=below:$x_1$] at (-1.3,0) (x1) {};
	\node [label=left:$x$] at (-1.3,1.2) (x) {};
	\node [label=left:$x^*$] at (-2,0) (x^*) {};
	\node [label=right:$y^*$] at (2,0) (y^*) {};
	\node [label=right:$y$] at (1.3,1.2) (y) {};
	\node [label=below:$y_2$] at (1.3,0) (y1) {};
	\node [label=below:$y_1$] at (.6,.1) (y2) {};
	\draw (x^*)--(x1)--(x)--(x2)--(u)--(y)--(y2)--(y1)--(y);  
	\draw (x1)--(x2);
	\draw (x)--(u)--(y2);
	\draw (y1)--(y^*);
	\scoped [on background layer] \draw (x^*) to [out=90,in=180] (0,1.5) to [out=0,in=90] (y^*);
	\scoped [on background layer] \fill [black,fill opacity=.2] (u.center)--(x2.center)--(x1.center)--(x^*.center)--(x^*.south) to [out=-90,in=180] (0,-1.5) to [out=0,in=-90](y^*.south)-- (y^*.center)--(y1.center)--(y2.center)--cycle;
	\draw [dashed] (x^*) to [out=-90,in=180] (0,-1.5) to [out=0,in=-90] (y^*);
	\end{tikzpicture} 
	\caption{The situation in the proof of Lemma~\ref{property}\ref{7-4}; $x_1, x_2, u, y_1, y_2$ are all distinct. The gray region has other vertices.}
	\label{B1-distinct} %
\end{figure}

\medskip\noindent
\ref{7-4}   (See Figure~\ref{B1-distinct}.)	If $|N^{\circ}(\{x_2,u\})-S|\ge 3$,
then deleting $x_2$, $u$ and collecting $x$, $y$ exposes $x_1$, $y_1$, $y_2$, and three other vertices and so $|B(G-\{x_2,u,x,y\})|\ge |B|-2+6$. By applying Lemma~\ref{cor-red} with $X=\{x_2,u\}$, $Y=\{x,y\}$, and $s=4$,  we obtain a contradiction. So we deduce that $|N^{\circ}(\{x_2,u\})-S|\le 2$. By symmetry, $|N^{\circ}(\{y_1,u\})-S|\le 2$.

\medskip\noindent
\ref{indu-} Suppose $x_1$ is adjacent to $u$. By \ref{xuy} and \ref{7-4}, $d(x_2)= 5$. Thus
$T:=ux_1x_2u$ is a separating triangle.  
Let $w_1$, $w_2$ be the two neighbours of $x_2$ other than $x_1$, $x$, $u$
so that $x_1w_1 w_2u$ is a path in $G$. Such a choice exists because
$G$ is a near plane triangulation.
 As $d(w_2)\ge 4$ by Lemma~\ref{d4}, and $u$ has no neighbours in $\inn(ux_1w_1w_2u)$ by \ref{7-4}, 
  $x_1$ is adjacent to $w_2$. 
 As $d(w_1)\ge 4$,  $T':=x_1w_1w_2x_1$ is a separating triangle. 
 Note that $x_2x_1w_1x_2$, $x_2w_1w_2x_2$, $x_2w_2ux_2$, and $ux_1w_2u$ are facial triangles. Thus $\|w_1,V(\ext(T'))\|=1$ and $\|w_2,V(\ext(T'))\|=2$, contrary to Lemma~\ref{sep-triangle}. 
So  $x_1$ is non-adjacent to $u$. 
By symmetry, $y_2$ is non-adjacent to $u$.

Suppose that $x_2$ is adjacent to $y_1$. Let $T'':=ux_2 y_1u$. By \ref{xuy}, $d(u)\ge 5$, so $T''$ is a separating triangle. 
By \ref{7-4}, $\|z,V(\inn(T''))\|\le2$ for all $z\in V(T'')$.
By Lemma~\ref{triangle}, 
\[\sum_{z\in V(T'')} \|z,V(\inn(T''))\|=\|V(T''),V(\inn(T''))\|\ge6\] 
and therefore $\|z,V(\inn(T''))\|=2$ for all $z\in V(T'')$.
By \ref{7-4}, $N(u)\cap V(\inn(T''))=N(x_2)\cap V(\inn(T''))=N(y_1)\cap V(\inn(T''))$.
Then $u$, $x_2$, $y_1$, and their neighbours in $\inn(T'')$ induce a $K_5$ subgraph, contradicting our assumption on $G$.
Thus $x_2$ is non-adjacent to $y_1$.

Suppose that $x_2$ is adjacent to $y_2$. Since $x_2$ is non-adjacent to $y_1$, \ref{xuy}  and \ref{7-4} imply that $d(y_1)=5$. 
Let $w_1$, $w_2$ be the two neighbours of $y_1$ other than $u$, $y$, $y_2$ such that $uw_1w_2y_2$ is a path in $G$.
By~\ref{7-4}, $N^\circ(u)-S\subseteq \{w_1,w_2\}$. If $u$ is adjacent to both $w_1$ and $w_2$, then $uw_1w_2u$, $uy_1w_1u$, $y_1w_1w_2y_1$ are facial triangles, implying that $w_1$ has degree $3$, contradicting Lemma~\ref{d4}. Thus, as $d(u)\ge 5$ by \ref{xuy}, we deduce that $d(u)=5$. Since $G$ is a near plane triangulation, $x_2$ is adjacent to $w_1$ and $ux_2w_1u$, $uw_1y_1u$ are facial triangles.
If $x_2w_1 w_2 y_2x_2$ is a separating cycle, then 
deleting $w_1$, $w_2$ and collecting $y_1$, $u$, $y$, $x$ exposes at least $4$ vertices
and so $\abs{B(G-\{w_1,w_2,y_1,u,y,x\})}\ge \abs{B}-2+4$.
By applying Lemma~\ref{cor-red} with $X=\{w_1, w_2\}$, $Y=\{y_1, u, y, x\}$, and $s=2$, we obtain a contradiction. 
So $x_2w_1 w_2 y_2x_2$ is not a separating cycle.
By Lemma~\ref{d4}, $d(w_2)\ge 4$ and therefore $w_2$ is adjacent to $x_2$ and 
  $d(w_1) =4= d(w_2)$. Then, deleting $y_1$ and collecting $w_1$, $w_2$, $u$, $y$, $x$ exposes $3$ vertices and $\abs{B(G-\{y_1,w_1,w_2,u,y,x\})}= \abs{B}-2+3$. By applying Lemma~\ref{cor-red} with  $X=\{y_1\}$, $Y=\{w_1, w_2, u, y,x\}$, and $s=1$, we obtain a contradiction. 
So $x_2$ is non-adjacent to $y_2$. By symmetry, $x_1$ is non-adjacent to $y_1$.

\medskip\noindent
\ref{7-7} Suppose that none of $x_2$, $u$, $y_1$ has neighbours in $B-\{x,y\}$.
By \ref{xuy}, \ref{7-4}, and \ref{indu-}, 
 $d(x_2)=5=d(y_1)$.  
 If \[|N^{\circ}(\{x_2,y_1\})-\{u\}|\ge5\] then 
 deleting $u$, $x_2$ and collecting $x$, $y$, $y_1$ exposes all vertices in $N^{\circ}(\{x_2,y_1\})-\{u\}$ and so 
 $\abs{B(G-\{u,x_2,x,y,y_1\})}\ge \abs{B}-2+5$.
 By applying Lemma~\ref{cor-red} with
  $X=\{u,x_2\}$, $Y=\{x,y,y_1\}$, and $s=3$, we  obtain a contradiction. Thus 
  $|N^{\circ}(\{x_2,y_1\})-\{u\}|\le 4$ and therefore   
  $x_2$, $y_1$ have the same set of neighbours in $V(G)-(B\cup S)$
  by \ref{dist} and \ref{indu-}. Let $w$, $w'$ be the neighbours of $x_2$ (and also of $y_1$) such that $w \in V(\inn(uy_1w'x_2u))$.
  Then $w$ is the unique common neighbour of $x_2$, $u$, and $y_1$. By \ref{7-4} and Lemma~\ref{d4}, $w$ is adjacent to $w'$. Thus $d(w)=4$.
  Deleting $u$ and collecting $w$, $x_2$, $y_1$, $x$, $y$ exposes at least $3$ vertices including $w'$ and so 
  $\abs{B(G-\{u,w,x_2,y_1,x,y\})}\ge \abs{B}-2+3$.
  This contradicts Lemma~\ref{cor-red} applied with $X=\{u\}$, $Y=\{w,x_2,y_1,x,y\}$, and $s=1$.

Thus at least one vertex of $x_2$, $u$, and $y_1$ is adjacent to a vertex in $B-\{x,y\}$. 
Then $x_1$ is non-adjacent to $y_2$. By \ref{indu-},   $G[S]$ is an induced path and \ref{7-7} holds.

\medskip\noindent 
\ref{7-6} Suppose that $x^*$ is adjacent to $x_2$.
 As $d(x_1) \ge 4$ by Lemma~\ref{d4}, $T:=x^*x_1 x_2 x^*$ is a separating triangle.
Since $d(x^*) \le 5$ by Lemma~\ref{deg<6}, $x^*$ has a unique neighbour $w\in V(\inn(T))$. So $w$ is adjacent to both $x_1$ and $x_2$.  As $d(w)\ge 4$ by Lemma~\ref{d4}, $T':=wx_1x_2w$ is a separating triangle with 
 $\|w, V(\ext(T'))\|=1$ and $\|x_1, V(\ext(T'))\|=2$,  
contrary to Lemma~\ref{sep-triangle}. So $x^*$ is non-adjacent to $x_2$. By symmetry, $y^*$ is non-adjacent to $y_1$.

Suppose $u$ is adjacent to $x^*$. 
As $d(x_1) \ge 4$ and $d(x^*) \le 5$ by Lemmas~\ref{d4} and \ref{deg<6}, $x^*$ has a unique neighbour $w\in V(\inn( x^*x_1x_2ux^*))$  adjacent to both $x_1$ and $u$.
By \ref{xuy} and \ref{7-4}, 
$w$ is adjacent to $x_2$.
If $uwx_2u$ is a separating triangle, then by Lemma~\ref{triangle}(d), $|N(\{x_2,u\})\cap V(\inn (uwx_2u))|\ge 2$, hence $|N^{\circ}(\{x_2,u\})-S|\ge 3$, contrary to \ref{7-4}. So $uwx_2u$ is facial.  
As $d(x^*) \le 5$,  $wx^*x_1w$  and $wx^*uw$ are  facial triangles. 
  As $d(x_2) \ge 5$ by \ref{7-4},   $T':=wx_1x_2w$ is a separating triangle.  So $\|x_1, V(\ext(T'))\|=2$ and $\|x_2, V(\ext(T'))\|=2$,   
contrary to Lemma~\ref{sep-triangle}.   Thus $u$ is non-adjacent to $x^*$. By symmetry, $u$ is non-adjacent to $y^*$. So \ref{7-6} holds.

   \medskip\noindent 
  \ref{7-8} Suppose that $w_1=y^*$. By \ref{7-6}, $y^*$ is adjacent to $x_2$. Let $C:= y^*x_2 u y_1 y_2 y^*$ and
  $C'$ be the cycle formed by the path from $x^*$ to $y^*$ in $\mathbf B(G)-x-y$ together with the path $y^* x_2x_1x^*$.
  Since $G$ is a near plane triangulation and $d(y_2)\ge 4$, by \ref{7-7}
  there is $w\in N(y^*)\cap N(y_2)\cap V(\inn(C))$. By Lemma~\ref{deg<6}, $d(y^*)=5$,
  and 
  therefore $x_2$ is adjacent to $w$ and $x_2wy^*x_2$ is a facial triangle.
  Let $y^{**}\in B$ be the neighbour of $y^*$ other than $y$.
  Then $x_2y^{**}y^*x_2$ is also a facial triangle in $G$. Because $x_2$ is non-adjacent to $x^*$ by \ref{7-6}, $y^{**}\ne x^*$. 
  By \ref{7-7} applied to $yy^*$, we have $y^*\in A$ because $uy_1y_2wx_2$ is not an induced path in $G$. Thus $x^*\notin A$ because $\abs{A}=2$. 
  By Lemma~\ref{nochord}, $\mathbf B(G)$ is chordless.
  Therefore by Lemma~\ref{d=5}, $d(x^*)=5$ and so $\|x^*,V(\inn(C'))\|=2$.
  By \ref{xuy}, \ref{7-4}, and \ref{indu-}, we have $\abs{N^\circ(y_1)- S}=2$.
  Deleting $x_1$, $u$ and collecting $x$, $x^*$, $y$, $y_1$ exposes at least $6$ vertices, including two neighbours of $x^*$ in $\inn(C')$ 
  and two neighbours of $y_1$ in $\inn(C)$. So $\abs{B(G-\{x_1,u,x,x^*,y,y_1\})}\ge \abs{B}-3+6$.
  By applying Lemma~\ref{cor-red} with $X=\{x_1,u\}$, $Y=\{x,x^*,y,y_1\}$, and $s=3$, we  obtain a contradiction. So $w_1\ne y^*$. By symmetry, $w_1\ne x^*$. Thus \ref{7-8} holds.
\end{proof}

\begin{lemma}
	\label{B1}
	$|B|=3$. 
\end{lemma}
 \begin{proof}  
For an edge $e=xy \in E(\mathbf{B}-A)$, let $x^*, x_1, x_2, u, y_1, y_2, y^*$ be as in Lemma~\ref{property}.
Suppose that $\abs{B}\ge 4$. Then $x^*\neq y^*$.
Lemma~\ref{property}\ref{7-8} implies that $B$ has a vertex other than $x$, $y$, $x^*$, and $y^*$. So, $\abs{B}\ge 5$.

We claim that $N^{\circ}(u) = \{x_2, y_1\}$.
Suppose not. 
By Lemma~\ref{cutvertex}, $|A|=2$, so at least one vertex of $\{x^*,y^*\}$, say, $y^*$ is not in $A$.
By Lemma~\ref{property}\ref{7-7} applied to $yy*$, 
we deduce that $u$ is non-adjacent to vertices in $N^\circ(y^*)$.
Thus deleting $u$, $y_2$ and collecting $y$, $x$, $y^*$ exposes at least 
$6$ vertices and so $\abs{B(G-\{u,y_2,y,x,y^*\})}\ge\abs{B}-3+6$. 
By applying Lemma~\ref{cor-red} with $X=\{u, y_2\}$, $Y=\{y,x,y^*\}$, and $s=3$, we obtain a contradiction. 
So $N^{\circ}(u) = \{x_2, y_1\}$.

Since $d(u) \ge 5$ by Lemma~\ref{property}\ref{xuy}, $u$ has at least one boundary neighbour $z \ne x,y$. Let $\mathbf B(x,z)$ be the boundary path from $x$ to $z$ not containing $y$, and $\mathbf B(y,z)$ be the boundary path from $y$ to $z$ not containing $x$. So $\mathbf B(x,z)$ and $\mathbf B(y,z)$ have only one vertex in common, namely $z$. One of $\mathbf B(x,z)$, $\mathbf B(y,z)$ has no internal vertex in $A$. We denote this path by $P(e, z)$. We choose $e=xy$ and $z$ so that $P(e,z)$ is shortest.  Assume $P(e,z) = \mathbf B(y,z)$. Let $e'=yy^*$. Then $e' \in E(\mathbf B-A)$. Let $y_2$ be the common neighbour of $y$ and $y^*$ and let $z' \ne y, y^*$ be a boundary neighbour of $y_2$.
Then $P(e', z')$ is a proper subpath of $P(e,z)$, and hence is shorter. This contradicts our choice of $e$ and~$z$.
\end{proof}

 \section{The final contradiction}\label{sec:proof}	
 In this section we complete the proof of Theorem~\ref{main}. First we prove a lemma.
 
\begin{figure}
	\centering
	\begin{tikzpicture}
	  \tikzstyle{every node}=[circle,draw,fill=black!50,inner sep=0pt,minimum width=4pt]
	  \draw (-30:2.5) node [label=right:$a'$] (a') {}
	  .. controls (30:1.2).. 
	  (90:1.3) node[label=$v$] (v){}
	  .. controls (150:1.2).. 
	  (210:2.5) node[label=left:$a$](a){}--cycle;
	  \draw (a)--(-.7,0) node[label=left:$x$](x){} -- (0,.7) node[label=below:$y$] (y){}
	  --(.7,0)node[label=right:$z$](z){}--(a');
	  \draw (v)--(x);
	  \draw (v)--(y);
	  \draw (v)--(z);
	  \scoped[on background layer] \fill [black,opacity=0.2](x.center)--(y.center)--(z.center)--(a'.center)--(a.center)--cycle;
	\end{tikzpicture}
	\caption{The situation of Lemma~\ref{fl}.}\label{fig:fl-setup}
  \end{figure}
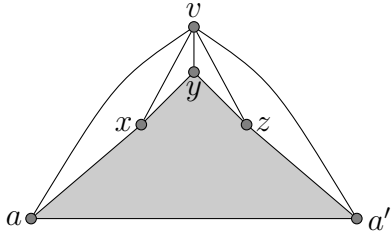
 
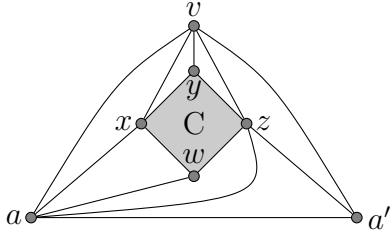
\begin{figure}
  \centering
  \begin{tikzpicture}
    \tikzstyle{every node}=[circle,draw,fill=black!50,inner sep=0pt,minimum width=4pt]
    \draw (-30:2.5) node [label=right:$a'$] (a') {}
    .. controls (30:1.2).. 
    (90:1.3) node[label=$v$] (v){}
    .. controls (150:1.2).. 
    (210:2.5) node[label=left:$a$](a){}--cycle;
    \draw (a)--(-.7,0) node[label=left:$x$](x){} -- (0,.7) node[label=below:$y$] (y){}
    --(.7,0)node[label=right:$z$](z){}--(a');
    \draw (x)--(0,-.7) node[label=$w$](w){}--(z) .. controls (1,-1) .. (a)--(w);
    \draw (v)--(x);
    \draw (v)--(y);
    \draw (v)--(z);
    \node [draw=none,fill=none] at (0,0) {C};
    \scoped[on background layer] \fill [black,opacity=0.2](x.center)--(y.center)--(z.center)--(w.center)--cycle;
  \end{tikzpicture}
  \caption{An illustration of the proof of Lemma~\ref{fl}\ref{za}.}\label{fig:fl}
\end{figure}

\begin{lemma}\label{fl}
If $B=\{a,a',v\}$ and $axyza'$ is a path in $G[N(v)]$ (see Figure~\ref{fig:fl-setup}), 
then the following hold. 
\begin{enum}
	\item\label{xz}$x$ is non-adjacent to $z$.
	\item \label{ya} $y$ is adjacent to neither $a$ nor $a'$. 
	\item \label{za} $z$ is non-adjacent to $a$ and $x$ is non-adjacent to $a'$.
	\item \label{xyz} $d(x), d(y), d(z)\ge 5$.
	\item \label{XYZ}  $|N^{\circ}(\{x,y,z\})|\le 4$.
	\item \label{xzw} $x$ and $z$ have a common neighbour $w \notin \{y, v\}$.   
	\item \label{xcapz} $N(x)\cap N(z)=\{v,w,y\}$.
	\end{enum}
\end{lemma}
\begin{proof}
	\ref{xz} Suppose $x$ is adjacent to $z$.
	As $K_{5}\nsubseteq G$, $x$ is non-adjacent to $a'$ or $z$ is non-adjacent to $a$; by symmetry, assume $x$ is non-adjacent to $a'$.
	Since $d(y) \ge 4$, $T:=xyzx$ is a separating triangle.  
	By Lemma~\ref{triangle}, $|V(\inn(T))| \ge 3$.
	Since $\|y, V(\ext(T))\|=1$, Lemma~\ref{sep-triangle} implies  $\|x, V(\ext(T))\|
	\ge 4$, 
	and so  $|N^{\circ}(x) \cap V(\ext(T))|
  \ge 2$. 
	
  If $d(y)\le6$, then deleting $x$, $z$ and collecting $v$, $y$ exposes at least $5$ vertices from $B(\inn(T))$ and $N^{\circ}(x)\cap V(\ext(T))$
  and so $\abs{B(G-\{x,z,v,y\})}\ge \abs{B}-1+5$. By applying Lemma~\ref{cor-red} with $X=\{x,z\}$, $Y=\{v,y\}$, and $s=4$, we obtain a contradiction. Therefore,
 $d(y)\ge7$. Then $|N^{\circ}(y)\cap V(\inn(T))|\ge 4$ and so 
 deleting $x$, $y$ and collecting $v$ exposes at least $7$ vertices,
 and $\abs{B(G-\{x,y,v\})}\ge \abs{B}-1+7$.
 By applying Lemma~\ref{cor-red} with $X=\{x,y\}$, $Y=\{v\}$, and $s=6$, we obtain a contradiction. So \ref{xz} holds.  

\medskip\noindent
	\ref{ya} Suppose $y$ is adjacent to $a$. Then  $T:=axya$ is a separating triangle, because $d(x) \ge 4$ and the other triangles incident with $x$ are facial. 
	As $d(a) \le 5$ by Lemma~\ref{deg<6}, $a$ has a unique neighbour $w$ in $\inn(T)$.
	As $d(w) \ge 4$, $T':=xwyx$ is a separating triangle. Now $\|w,V(\ext(T'))\|=1$, and $\|x,V(\ext(T'))\|=2$,   
	contrary to Lemma~\ref{sep-triangle}. Thus $y$ is non-adjacent to $a$. By symmetry, $y$ is non-adjacent to $a'$.  So \ref{ya} holds.
	
	\medskip\noindent 
	\ref{za} Suppose that $z$ is adjacent to $a$.
	By \ref{xz}, $z$ is non-adjacent to $x$. As $d(x) \ge 4$ and $d(a) \le 5$ by Lemmas~\ref{d4} and \ref{deg<6},  there is $w\in (N(a)\cap N(x)\cap N(z))-\{v\}$, and $xawx$, $wazw$, $aza'a$ are all facial triangles. (See Figure~\ref{fig:fl}.) By \ref{ya}, $y\ne w$. 
	Since $d(y)\ge 4$ by Lemma~\ref{d4}, $C:=xyzwx$ is a separating cycle of length $4$. Let $I=\inn(C)$.  
	Then $V=B\cupdot V(C)\cupdot V(I)$, 
	\begin{enum*}
	\item\label{e1} $\|x,V(\ext(C))\|=2$, \item\label{e2} $\|y,V(\ext(C))\|=1$, \item\label{e3} $\|z,V(\ext(C))\|=3$, and \item\label{e4}$\|w,V(\ext(C))\|=1$.
	\end{enum*}

 If $w$ is adjacent to $y$, then we apply Lemma~\ref{cor-reducible} with $C$, $X=\{w\}$, and $Y=\emptyset$. As $y$ is adjacent to $w$, $A':=\{x,y,z\}$ is usable in $G_1:=\inn[C]-w$, and by (i--iii), $A'$ is collectable in $G'_2:=\ext[C]-w$.  As $\abs{V(G_2)}=\abs{B}=3$, $\tau(G_2)=0$.  This contradicts Lemma~\ref{cor-reducible}.
 
	So  using \ref{xz}, $C$ is chordless and $x$ has at least one neighbour in $\inn(C)$.

By Lemma~\ref{C4septri}, either $|B(I)|\ge 4$ or $\abs{V(I)}\le 2$ and every vertex in $I$ has degree $4$ in $G$.
If $\abs{V(I)} \le 2$ and every vertex in $I$ has degree $4$ in $G$, then  $V-\{x\}$ is $A$-good as we can collect $V(I),y,w,z,v,a',a$.
Then $f(G;A)\ge \abs{V(G)}-1 \ge \partial(G)$, a contradiction.  
Therefore  $|B(I)|\ge 4$. 

If there is an edge $uu'\in E(C)$ with $|N(\{u,u'\})\cap V(I)|\ge4$,
then we apply Lemma~\ref{cor-reducible} with $C$, $X=\{u,u'\}$ and $Y=\emptyset$.  Now $A':=V(C)-\{u,u'\}$ is usable in $G_1:=\inn[C]-\{u,u'\}$, $A'$ is collectable in $G_2':=\ext[C]-\{u,u'\}$, $|B_1|\ge6$, and $B_2=B$. As $G_2=\mathbf B$, $\tau(G_2)=0$. This contradicts Lemma~\ref{cor-reducible}.
So 
$|N(\{u,u'\})\cap V(I)|\le3$ for all edges $uu'\in E(C)$ and in particular,
$\|u,V(I)\|\le3$ for all $u\in V(C)$. This implies $d(y)\le 6$.

If $|N(\{x,y,z\})\cap V(I)|\ge4$, then we apply  Lemma~\ref{cor-reducible} with $C$, $X=\{x,z\}$ and $Y=\{v,y\}$.  Then $Y$ is collectable in $G-X$, $A':=\{w\}$ is usable in $G_1:=\inn[C]-\{x,y,z\}$, $A'$ is collectable in $G_2':=\ext[C]-\{x,y,z\}$, $|B_1|\ge5$, and $B_2=B-\{v\}$. 
As $(X \cup Y)\cap B\neq\emptyset$, this contradicts Lemma~\ref{cor-reducible}.

Therefore $|N(\{x,y,z\})\cap V(I)|\le 3$. 
Since $\abs{B(I)}\ge 4$, there exists a vertex $u$ in $B(I)-N(\{x,y,z\})$. 
Then $w$ is the only neighbour of $u$ in $C$.

Because $G$ is a plane triangulation and $d(u)\ge 4$, $w$ is adjacent to $u$.
Since $u$ is non-adjacent to $x$, $y$, $z$, we deduce that 
$B(I)\cap N(w)$ contains $u$ and at least two of the neighbours of $u$.
Since $\| w, V(I)\|\le 3$, we deduce that $\|w,V(I)\|=3$.
Since $|N(\{x,w\})\cap V(I)|\le 3$, 
all neighbours of $x$ in $I$ are adjacent to $w$. 
Similarly all neighbours of $z$ in $I$ are adjacent to $w$.
Since $\abs{B(I)}\ge 4$, there is a vertex $t$ in $B(I)$ non-adjacent to $w$.
Then $t$ is non-adjacent to $x$ and $z$.
Therefore $t$ is adjacent to $y$.
By the same argument, $\| y, V(I)\| = 3$ and every neighbour of $x$ or $z$ in $I$
is adjacent to $y$.
Thus, every vertex in $N(\{x,z\})\cap V(I)$ is adjacent to both $y$ and $w$.

If $\|x,V(I)\|\ge 2$, then $x$, $y$, $w$, and their common neighbours in $I$ together with $a$ are the branch vertices of a $K_{3,3}$-subdivision, using the path $avy$. So $G$ is nonplanar, a contradiction. Thus, $\|x,V(I)\|\le 1$ and similarly $\|z,V(I)\|\le 1$.
This means that $d(x)\le 5$ and $B(\inn[C]-\{x,y,w\})=B(I)\cup \{z\}$.

We apply Lemma~\ref{cor-reducible} with $C$, $X=\{w,y\}$ and $Y=\{x\}$.   Then $Y$ is collectable in $G-X$ and $A'=\{z\}$ is usable in $G_1:=\inn[C]-\{w,x,y\}$, $A'$ is collectable in $G'_2:=\ext[C]-\{w,x,y\}$, $|B_1|=\abs{B(I)\cup\{z\}}\ge5$, $B_2=B$, and $G_2=\mathbf B$. Thus $\tau(G_2)=0$ and this contradicts Lemma~\ref{cor-reducible}.  
Hence  $z$ is non-adjacent to $a$. By symmetry, $x$ is non-adjacent to $a'$. Thus \ref{za} holds.

\medskip\noindent 
\ref{xyz}	Suppose $d(u)\le4$ for some $u\in\{x,y,z\}$. 
  By Lemma~\ref{d4}, $d(u)=4$. Let $u':=y$ if $u\ne y$,  $u':=x$ otherwise.
  Then, deleting $u'$ and collecting $u$, $v$ exposes at least $2$ vertices in $N^\circ(\{u,u'\})$ by \ref{xz} and \ref{za} and so $\abs{B(G-\{u,u',v\})}\ge \abs{B}-1+2$.
  By applying Lemma~\ref{cor-red} with $X=\{u'\}$, $Y=\{u,v\}$, and $s=1$, we obtain a contradiction. So \ref{xyz} holds.

\medskip\noindent 
\ref{XYZ} Suppose $|N^{\circ}(\{x,y,z\})|\ge 5$. If $d(y)\le6$, then 
deleting $x$, $z$ and collecting $v$, $y$ exposes at least $5$ vertices and so 
$\abs{B(G-\{x,z,v,y\})}\ge \abs{B}-1+5$. 
By applying Lemma~\ref{cor-red} with $X=\{x,z\}$, $Y=\{v,y\}$, and $s=4$, 
we obtain a contradiction.   
	Thus $d(y)\ge7$. Then either $|N^{\circ}(\{x,y\})-\{z\}|\ge5$ or $|N^{\circ}(\{z,y\})-\{x\}|\ge5$. We may assume by symmetry that $|N^{\circ}(\{x,y\})-\{z\}|\ge5$. 
  Then deleting $x$, $y$ and collecting $v$ exposes at least $6$ vertices and so $\abs{B(G-\{x,y,v\})}\ge \abs{B}-1+6$.
  By applying Lemma~\ref{cor-red} with $X=\{x,y\}$, $Y=\{v\}$, and $s=5$, 
	we obtain a contradiction.	So \ref{XYZ} holds.

\begin{figure}
	\centering
	\begin{tikzpicture}
	\tikzstyle{every node}=[circle,draw,fill=black!50,inner sep=0pt,minimum width=4pt]
	\node [label=$v$] at (0,2) (v) {};
	\node [label=left:$a$] at (-3,-2) (a) {};
	\node [label=right:$a'$] at (3,-2) (a') {};	
	\node [label=left:$x$] at (-1,0) (x) {};
	\node [label=below:$y$] at (0,0) (y) {};
	\node [label=right:$z$] at (1,0) (z) {};
	\node [label=below:$x_1$] at (-1.2,-1) (x1) {};
	\node [label=below:$x_2$] at (-.5,-1) (x2) {};
	\node [label=below:$z_1$] at (.5,-1) (z1) {};
	\node [label=below:$z_2$] at (1.2,-1) (z2) {};
    \draw (v)--(a)--(a')--(v);
    \draw (a)--(x)--(v)--(y)--(z)--(v);
    \draw (x)--(y);
    \draw (z)--(a');
    \draw (a)--(x1)--(x)--(x2)--(y)--(z1)--(z)--(z2)--(a');
    \draw (x1)--(x2);
    \draw (z1)--(z2);	
    \draw (x2)--(z1);
    \scoped[on background layer]
    \fill [black,opacity=.2,pattern=north east lines](a.center)--(x1.center)--(x2.center)
    --(z1.center)--(z2.center)--(a'.center)--cycle;
    \node [draw=none,fill=none] at (0,-1.5){$C^*$};
	\end{tikzpicture} 
	\caption{Proof of Lemma~\ref{fl}\ref{xzw}. There are no vertices in $\inn(C^*)$.}
	\label{figurecom.nei.ofxz} %
\end{figure}
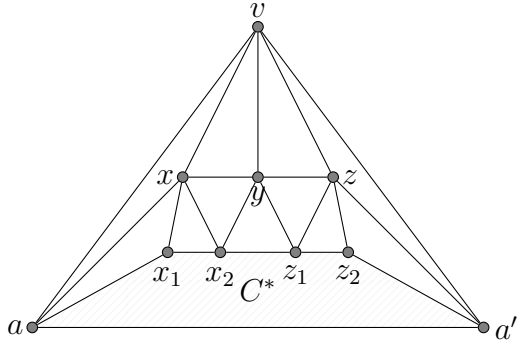

\medskip\noindent 
\ref{xzw}	Suppose $N(x)\cap N(z)=\{y,v\}$. By \ref{xyz}, $d(x), d(z) \ge 5$. By \ref{XYZ}, $|N^{\circ}(\{x,z\})-\{y\}|\le 4$. 
By \ref{za}, $z$ is non-adjacent to $a$
and $x$ is non-adjacent to $a'$
and by \ref{xz}, $x$ is non-adjacent to $z$.
So each of $x$ and $z$ have exactly two neighbours in $\inn(axyza'a)$ and $d(x)=d(z)=5$. Let $x_1$, $x_2$ be those neighbours of $x$
and $z_1, z_2$ be those two neighbours of $z$.
We may assume that $x_1x_2yz_1z_2$ is a path in $G$
by swapping labels of $x_1$ and $x_2$ and swapping labels of $z_1$ and $z_2$ if necessary. 
By \ref{XYZ}, 
we have $N^{\circ}(y)-\{x,z\}\subseteq \{x_1,x_2,z_1,z_2\}$. As $d(x_2)\ge 4$, $y$ is not adjacent to $x_1$ because otherwise $x_1x_2yx_1$ is a separating triangle, that will make a new interior  neighbour of $y$ by Lemma~\ref{triangle}(c), contrary to \ref{XYZ}.
By symmetry, $y$ is not adjacent to $z_2$. So $x_2$ is adjacent to $z_1$ as $G$ is a plane triangulation. 
Therefore $d(y)=5$.

Let $C^*:=ax_1x_2z_1z_2a'a$. Suppose that  $w\in N(\{x_1,x_2,z_1,z_2\})\cap V(\inn (C^*))$. Then by symmetry, we may assume $w$ is adjacent to $x_1$ or $x_2$. 
Deleting $x_1$, $x_2$ and collecting $x$, $y$, $v$, $z$ exposes $w$, $z_1$, $z_2$ and so $\abs{B(G-\{x_1,x_2,x,y,v,z\})}\ge \abs{B}-1+3$.
By applying Lemma~\ref{cor-red} with $X=\{x_1,x_2\}$, $Y=\{x,y,v,z\}$, and $s=2$,  
we obtain a contradiction. Thus $N(\{x_1,x_2,z_1,z_2\})\cap V(\inn (C^*)) = \emptyset$ and therefore  $\abs{G}=10$. See Figure~\ref{figurecom.nei.ofxz}.

By Observation~\ref{ob0} applied to $\inn[C^*]$, there is a vertex $w\in\{x_1,x_2,y_1,y_2\}$ having degree at most $2$ in $\inn[C^*]$.
By symmetry, we may assume that $w=x_i$ for some $i\in \{1,2\}$.
Since $d(x_i)\le 4$, 
after deleting $x_{3-i}$, we can collect $x_i$, $x$, $y$, $v$, $z$, resulting  in an outerplanar graph, which can be collected by Observation~\ref{ob0}.
So, $f(G;A)\ge9\ge \partial(G)$, a contradiction. So \ref{xzw} holds.

\medskip\noindent 
\ref{xcapz} Suppose there is $w'\in N(x)\cap N(z)-\{v,w,y\}$. 
Let $C:=xyzwx$.
We may assume that $w$ is chosen to maximize $\abs{V(\inn(C))}$.
So $w'$ is in $V(\inn(C))$ and together with~$\ref{xz}$, we deduce that $C$ is an induced cycle.
 
 We claim that $y$ is non-adjacent to $w'$.
 Suppose not.  As $d(y)\ge 5$ by \ref{xyz}, $xw'yx$ or $zw'yz$ is a separating triangle. By symmetry, we may assume $xw'yx$ is a separating triangle. Thus $|N(\{x,y\})\cap V(\inn (xw'yx))|\ge 2$ by Lemma~\ref{triangle}(d). Because $G$ is a plane triangulation, by \ref{XYZ}, $w$ is adjacent to $w'$ and $xww'x$, $zww'z$, and $yzw'y$ are facial triangles. Thus $\|y,V(\ext(xw'yx))\|=\|w',V(\ext(xw'yx))\|=2$, contrary to Lemma~\ref{sep-triangle}.  This proves the claim that $y$ is non-adjacent to $w'$.
 
 Therefore
 $\|y,V(\inn(xyzw'x))\|=2$ by \ref{xyz} and \ref{XYZ}.  Let $y_1$, $y_2$ be two neighbours of $y$ in $\inn(xyzw'x)$ such that $xy_1y_2z$ is a path in $G$. Because $G$ is a	plane triangulation, by \ref{XYZ}, $w'$ is adjacent to both $y_1$ and $y_2$
 and $\inn(xw'zwx)$ has no vertex. Then 
 $C$ is a separating induced cycle of length $4$ and $|B(\inn(C))|=3$, contrary to Lemma~\ref{C4septri}.	So \ref{xcapz} holds. 
\end{proof}

\begin{proof}[Proof of Theorem~\ref{main}] 
Let $(G;A)$ be an extreme counterexample. 
Then $G$ is a near plane triangulation.
Let $B=B(G)$ and $\mathbf B=\mathbf B(G)$.
By Lemmas~\ref{cutvertex} and \ref{B1}, $\abs{B}=3$ and $\abs{A}=2$. Let $A=\{a,a'\}$ and $v\in B-A$. 
By Lemma~\ref{d=5}, $d(v)=5$. 
As $G$ is a plane triangulation, 
the neighbours of $v$ form a path $axyza'$.
By Lemma~\ref{fl}\ref{xcapz}, $x$ and $z$ have exactly one common neighbour $w$ in $G-v-y$.  
Then $C:=xyzwx$ is a cycle of length $4$. 
By symmetry and Lemma~\ref{fl}\ref{xyz}, we may assume that $d(x)\ge d(z)\ge5$.  
By Lemma~\ref{fl}\ref{XYZ}, 
\[
(d(x)-3)+(d(z)-3)-1 \le 
|N^\circ(\{x,y,z\})|
\le 4.
\]
Therefore  $d(z)= 5$ and $d(x)=5$ or $6$.

We claim that $y$ is non-adjacent to $w$. Suppose that $y$ is adjacent to $w$. 
By Lemma~\ref{fl}\ref{xyz}, $d(y)\ge 5$ and therefore 
at least one of $xywx$ and $yzwy$ is a separating triangle. If both of them are separating triangles, then  $|N(\{x,y\})\cap V(\inn (xywx))|\ge 2$ and $|N(\{y,z\})\cap V(\inn (yzwy))|\ge 2$, by Lemma~\ref{triangle}(d).  Therefore $|N^{\circ}(\{x,y,z\})|\ge 2+2+1=5$, contrary to Lemma~\ref{fl}\ref{XYZ}.
This means that exactly one of $xywx$ and $yzwy$ is a separating triangle.

Suppose $yzwy$ is a separating triangle.
Then $xywx$ is a facial triangle, and $z$ has a neighbour in $\inn(yzwy)$.  
As $d(z)=5$, $z$ has no neighbour in $\inn(axwza'a)$. Therefore, 
$w$ is adjacent to $a'$, and $wza'w$ is a facial triangle.  
Thus $\|y,V(\ext(yzwy))\|=\|z,V(\ext(yzwy))\|=2$, contrary to Lemma~\ref{sep-triangle}. 
So $yzwy$ is not a separating triangle.

Therefore $xywx$ is a separating triangle.
By Lemma~\ref{triangle}(d), $\inn(xywx)$ has at least two vertices in $N^\circ(\{x,y,z\})$.
By Lemma~\ref{fl}\ref{xyz}, 
$z$ has a neighbour in $\inn(axwza'a)$.
Then already we found four vertices in 
$N^\circ (\{x,y,z\})$.  
This means that $x$ has no neighbours in $\inn(axwza'a)$ by Lemma~\ref{fl}\ref{xzw}.
Hence $\|y,V(\ext(xywx))\|=\|x,V(\ext(xywx))\|=2$, contrary to Lemma~\ref{sep-triangle}. 
This completes the proof of the claim that $y$ is non-adjacent to $w$. 

Therefore $C$ is chordless by Lemma~\ref{fl}\ref{xz}. 
 By Lemma~\ref{fl}\ref{xyz},  $d(y)\ge 5$. Thus $C$ is a separating induced cycle of length $4$. 
  By Lemma~\ref{C4septri}, either $|B(\inn(C))|\ge 4$ or both $\abs{V(\inn(C))}\le 2$ and every vertex in $\inn(C)$ has degree $4$ in $G$.

If $|B(\inn(C))|\ge 4$, then  
deleting $w$, $y$ and collecting $z$, $v$, $x$ exposes at least $4$ vertices
and so $\abs{B(G-\{w,y,z,v,x\})}=\abs{B}-1+4$.
By applying Lemma~\ref{cor-red} with $X=\{w,y\}$, $Y=\{z,v,x\}$, and $s=3$, 
we obtain a contradiction.

Therefore  $1\le |V(\inn(C))| \le 2$ 
and every vertex in $\inn(C)$ has degree $4$ in $G$.   
Deleting $y$ and collecting all vertices in $\inn(C)$ and $z$, $v$, $x$
exposes $w$ and so $B(G-(\{y,z,v,x\}\cup V(\inn(C))))\ge \abs{B}-1+1$.
By applying Lemma~\ref{cor-red} with $X=\{y\}$, $Y=V(\inn(C))\cup\{z,v,x\}$, and $s=0$,  
we obtain a contradiction. 
	 \end{proof}  

	 \providecommand{\bysame}{\leavevmode\hbox to3em{\hrulefill}\thinspace}
	 \providecommand{\MR}{\relax\ifhmode\unskip\space\fi MR }
	 \providecommand{\MRhref}[2]{%
	   \href{http://www.ams.org/mathscinet-getitem?mr=#1}{#2}
	 }
	 \providecommand{\href}[2]{#2}

\end{document}